\documentclass{nhm}
\usepackage[latin1]{inputenc}
\usepackage[T1]{fontenc}
 \usepackage{color}
\usepackage{subfigure}
 
\def\currentvolume{X}
 \def\currentissue{X}
  \def\currentyear{200X}
   \def\currentmonth{XX}
    \def\ppages{X--XX}

\usepackage[dvips]{graphics}
\usepackage{rotate}
\usepackage{epsf}
\usepackage{graphicx}
\usepackage{epsfig} 
\usepackage{amsmath}
\usepackage[psamsfonts]{amssymb}
\usepackage[psamsfonts]{eucal}
\usepackage{psfrag}

\newtheorem{theorem}{Theorem}[section]

\newtheorem{proposition}{Proposition}

\theoremstyle{definition}
\newtheorem{definition}[theorem]{Definition}

\def \endproof {\phantom {.} \hfill $\square$}
\def \qq   {{\bf q}}
\def \vv   {{\bf v}}
\def \ww   {{\bf w}}
\def \uu   {{\bf u}}
\def \UU   {{\bf U}}

\def \GG   {{\bf G}}
\def \lcal {\mathcal{L}}
\def \PP {\mathcal{P}}
\def \setR {\mathbb{R}} 
 
\def \set#1{\left\{{#1}\right\}} 
\def \abs#1{\left | {#1}\right | }  
\def \virg {\, , \,\,} 
\def \argmin  {\mathop{\textmd{ argmin }}}
\def \supp  {\mathop{supp}}
\def \para #1 { \noindent {\bf #1. }}
\def \vseq {\vspace{6pt}}
\def \BE {\begin{equation}}
\def \EE {\end{equation}}
\def \OO {\Omega}

\def \bft {{\bf t}}
\def \ONE {\mathbf{1}}

\def \ee  {{\bf e }}

\def \nn  {{\bf n}}
\def \dd  {{\bf d}}

\newcommand{\eps}{\varepsilon}

\title[Handling congestion in crowd motion modeling
]
      {
      Handling congestion in crowd motion modeling}

\author[B. Maury, A. Roudneff-Chupin, F. Santambrogio and J. Venel]{}


\email{Bertrand.Maury@math.u-psud.fr}
\email{Aude.Roudneff@math.u-psud.fr}
\email{Filippo.Santambrogio@math.u-psud.fr}
\email{Juliette.Venel@univ-valenciennes.fr}

\begin{document}
\maketitle





\centerline{\scshape B. Maury, A. Roudneff-Chupin, F. Santambrogio}
\medskip
{\footnotesize
 \centerline{Laboratoire de Math\'ematiques d'Orsay, Universit\'e Paris-Sud,}
   \centerline{91405 Orsay Cedex, France}
} 

\medskip

\centerline{\scshape J. Venel }
\medskip
{\footnotesize
 \centerline{LAMAV, Universit\'e de Valenciennes et du Hainaut-Cambr\'esis Mont Houy}
   \centerline{59313 Valenciennes Cedex 9, France}
} %

\bigskip


\begin{abstract}
We address here the issue of congestion in the modeling of crowd
motion, in the non-smooth framework: contacts between people are not
anticipated and avoided, they actually occur, and they are explicitly taken into account in the model. 
We limit  our approach  to  very basic principles in terms of   behavior,   to focus on the particular problems raised by the non-smooth character of  the models.  We consider that individuals tend to move according to a desired, or spontanous, velocity. We account for congestion by assuming that the evolution realizes at each time  an instantaneous balance between individual tendencies and global constraints (overlapping is forbidden):
the actual velocity  is defined as the closest to the desired velocity among all admissible ones, in a least square sense. 
We develop those principles in the microscopic and macroscopic settings, and we present  how the framework of Wasserstein distance between measures allows to recover  the {\em sweeping process} nature of the problem on the macroscopic level, which makes it possible to obtain existence results in spite of the non-smooth character of the evolution process. Micro and macro approaches are compared, and we investigate the similarities  together with deep differences  of those two levels of description.

\end{abstract}


\section{Introduction}

Congestion phenomena in population dynamics cover a wide range of mechanisms, which can be classified according  to their stiffness~:
\begin{enumerate}
\item Soft congestion: as the distance between individuals becomes smaller, the behavior of a single person is affected by the presence of others;

\item Hard congestion: actual contacts between individuals occur, and the overall motion  is perturbed by the fact that two persons may not occupy the same place at the same time. 

\end{enumerate}


 Soft congestion models can be included in both microscopic and macroscopic models. 
In the first class there are discrete-space models like cellular
automata-based models which constrain pedestrians to be located at
squares of a fixed grid~\cite{Blue,Burstedde,Nagel,Schad, Schadant}. We
can also mention models based on networks as route choice
models~\cite{BorgersTim0, BorgersTim} or queuing models~\cite{Lovas,
  Yuhaski}. A lot of evacuation softwares rest on such models
as for example buildingExodus in ~\cite{Exodus}.  
Some microscopic models are space continous as the social force model introduced in~\cite{Helb3} and its forerunner which has been proposed by~\cite{GippsMark}. In~\cite{Helb3}, people are identified to particles submitted to the laws of Newtonian mechanics.

As for macroscopic models, soft congestion is also usually favored. Some of these models (see for example~\cite{Col},~\cite{Cha1} and~\cite{Cha2}) own their origins in vehicular traffic models, and deal with the one-dimensional case. In higher dimension, similarly to the microscopic case, many models rely on social forces. One solution, used e.g. by Bellomo and Dogbe in~\cite{Bel} and~\cite{Dog}, or Degond in~\cite{Deg}, is to add to the equation satisfied by the velocity a repulsive term that forces people to avoid high density areas. Another possibility is to modify directly the velocity by adding a term that make people deviate from their preferred path as soon as they approach a  crowded area. The modeling of the velocity is the key point of these strategies, see for example the work of Hughes~\cite{Hug1},~\cite{Hug2}, Coscia~\cite{Cos}, or Piccoli~\cite{Pic1},~\cite{Pic2}.

\medskip

We aim here at addressing the particular issues  pertaining to {\em hard} congestion. In this spirit, we shall make very crude assumptions on the  social aspects triggering crowd dynamics (see Section~\ref{sec:ext} for possible improvements on these aspects), and simply consider that, at any time, a spontaneous (or desired) velocity field is given, and that it is purely {\em reptilian}: everyone's  priority is to achieve his own goals, without accounting for others. Actual behavior shall result from some sort of compromise between individual tendencies and congestion constraints. 
More precisely, 
we shall consider that the actual velocity is defined as the projection of the spontaneous one onto the set of feasible velocities (i.e.\ which do not lead to a violation of the non-overlapping constraint).

Those basic principles can be applied to microscopic and macroscopic descriptions of the crowd. In the microscopic setting, the degrees of freedom are  the positions of individuals (identified to rigid disks), and the non-overlapping constraints can be written straightforwardly by prescribing a minimal value for the distance between centers. 
The problem takes the form of a differential inclusion which fits into the general framework of sweeping processes introduced by Moreau in the 70's (see~\cite{Mor}). 
This approach has been extended more recently to non convex cases, namely the case of
uniformly prox-regular sets in \cite{Colombo}, in \cite{Thibsweep} and later in \cite{Monteiro}. The perturbed (i.e. inhomogeneous) problem has been studied in \cite{Thibnonconv, Thibsweep, Thi1, Thi2}. As we shall see, it gives a natural framework for the microscopic model  we propose. 

In the macroscopic setting, the crowd is seen as a population density and the non-overlapping constraint consists in prescribing a maximal value for this density.
Although both microscopic and macroscopic models express the same type of modeling assumptions, the mathematical structure of the macroscopic model is less obvious. Expressed in the Eulerian framework as a transport equation by the actual velocity field, which is defined as the projection of the spontaneous one onto the set of feasible velocities, it involves a conservative transport equation by a field  whose regularity cannot be  controlled a priori,  so that classical results cannot be used. A first attempt to address those issues was proposed in~\cite{MRS10}, in the case where the spontaneous velocity is the gradient of a given dissatisfaction function (typically the distance to the exit for the evacuation of a building):  the framework of optimal transportation allows to reformulate the problem as a gradient flow in the Wasserstein space of measures (see e.g. \cite{Vil1} for an overview of the theory of optimal transport, and \cite{Amb} for the notion of gradient flow in
  this setting).  We show here that the sweeping process framework, which does not rely on any  assumption on the gradient nature of the forcing term, can be extended in the macroscopic situation.

Stemming from similar principles, both microscopic and macroscopic models present nevertheless deep differences. For instance we shall point out the fact that the notion of maximal  density is not properly  defined  in the microscopic setting. Thus the macroscopic one cannot be expected to be obtained from the microscopic one by any homogenization process.

We shall focus here on the particular issues related to hard congestion. In order to highlight the very  problems it raises, we will favor a very basic form of the crowd motion model. In particular we shall disregard in the main part of this paper  any social or strategical aspects of human behavior. 
We are aware of the very crude character of these assumptions, and we gathered in Section~\ref{sec:ext}  other aspects of pedestrian behavior that could be included in our approach. 
Besides, we hope that the framework presented here, and the issues we raise in attempting to establish links between microscopic and macroscopic settings, may be fruitful in other contexts, e.g.  micro-macro issues in granular flows or modeling of chemotactic motion of large populations of cells. 

\section{Hard congestion models}

\subsection{Microscopic setting}

We consider  a population of $N$ individuals, identified with rigid disks centered at $\qq_1$, $\qq_2$, \dots , $\qq_N$, with common radius $r>0$ (it can be extended straighfowardly to the so-called {\em polydisperse} case, i.e. with different radii).   
We denote by 
$$
\UU(\qq) = (\UU_1(\qq),\dots,\UU_N(\qq))
$$
 the generalized spontaneous velocity of the crowd ($\UU_i$ is the velocity which the individual $i$ would like to have in the absence of others). We assume here that the desired velocity of $i$ depends on his  location only, and  that  this dependence  is the same for all individuals (see Section~\ref{sec:ext} for more general behavioral models):
$$
\UU_i = \UU_0(\qq_i),
$$
where $\UU_0$ is a given field. 
%
%
The set of feasible configurations is defined as
$$
K = \set{\qq \in \setR^{2N}\virg D_{ij}(\qq) = \abs {\qq_j - \qq_i } - 2 r \geq 0 \virg \forall i \neq j}.
$$
We disregard here obstacles (like walls, furniture) to alleviate notations, but they can be included in the definition of $K$ straightforwardly.
Non-overlapping is preserved by prescribing that  $D_{ij}$ may not decrease if it is $0$, i.e.\
$\dot D_{ij} = \GG_{ij}(\qq)\cdot \dot \qq\geq 0 $ where $\GG_{ij}(\qq) = \nabla D_{ij}$. It 
 leads to define the set of feasible velocities as
\BE
\label{eq:Cqdef}
C_\qq = \set{\vv \in \setR^{2N}\virg D_{ij}(\qq)=0 \Longrightarrow \GG_{ij}(\qq)\cdot \vv \geq 0 \ }.
\EE
The model simply writes
\BE
\label{eq:pbmic}
\frac {d\qq}{dt } = P_{C_\qq} \UU(\qq)
\EE
where the projection is performed according to the euclidean norm over $\setR^{2N}$.\\

\para{Saddle-point formulation}
Problem (\ref{eq:pbmic}) is self consistent and drives the evolution process, as will be shown in Section~\ref{eq:pbmic}. 
Yet, the saddle-point formulation of the projection problem sheds light on the underlying Darcy-like structure of the problem, and introduces a pressure field which will have a straightforward interpretation in terms of modeling, and which will have a continuous counterpart in the macroscopic setting. 

At some given configuration $\qq$, let us denote by $\Lambda$ the set of couples $(i,j)$, $i<j$, such that $i$ and $j$ are in contact. 
Constraints on the velocity can be written in a  matrix form (we drop the dependence of matrix $B$ upon the configuration)
$$
B \vv \leq 0,
$$
where the rows of $B$ correspond to active constraints :
$$
\GG_{ij}\cdot \vv \geq 0 \quad (i,j) \in \Lambda.
$$
The saddle-point formulation of the projection problem consists in finding $(\uu,p)\in \setR^{2N} \times \setR_+^{N_\Lambda}$ (where $N_\Lambda$ is the number of active constraints), such that
\BE
\label{eq:spmic}
\left \{
\begin{array}{lcl}
\uu + B^\star p &=& \UU \vseq \\
B \uu &\leq& 0.
\end{array}
\right . 
\EE
supplemented by the complementarity condition 
$$
p\cdot B\uu = 0.
$$
Pressure $p_{ij}$ can be seen as the interaction force between individuals $i$ and $j$.

\subsection{Macroscopic setting}

In the macroscopic setting the crowd is represented by a density $\rho$, which we shall assume to be supported within some domain (the room) $\overline \Omega$.  In order to alleviate the notations, we shall present the model in the case of a closed room (see the end of Section~\ref{sec:theomac} for some details on the way an exit door can be accounted for in this framework).
If we set at $1$ the saturation value, the set of feasible densities is 
\begin{equation}\label{definition K}
K = \set{ \rho \in L^1(\setR^2)\virg \int \rho(x)\, dx = 1\virg \supp (\rho) \subset \overline{\Omega}\virg 0\leq \rho (x) \leq 1\hbox{ for a.e.  } x}  \, .
\end{equation}
The density is advected by the actual velocity $\uu$ (macroscopic counterpart of the microscopic one $d\qq / dt $)
$$
\partial_t \rho  + \nabla \cdot (\rho \uu) = 0,
$$
and $\uu$ is defined as
$$
\uu = P_{C_\rho}\UU,
$$
where $C_\rho$ is the set of feasible velocities. It can be defined unformally as the set of all those velocities which have  a non-negative divergence on the saturated zone $[\rho=1]$ (where the density  may not increase). More precisely, it is defined by duality as 
\begin{equation}
\label{eq:defCrho}
C_\rho = \left \{ \mathbf{v} \in L^2(\Omega)^2, \; \int_\Omega \mathbf{v} \cdot \nabla q \leq 0 \quad \forall q \in   H^1_+(\Omega)\virg q(x) = 0 \hbox { a.e.  on } [\rho<1] 
\right \},
\end{equation}
$$
\hbox{ with }H^1_+(\Omega) = \{ q\in H^1(\Omega)\, ,\,\, q \geq 0 \hbox { a.e. in  }\Omega
 \}. $$

\vspace*{0.2cm}
\para{Saddle-point formulation}
The dual expression of $C_\rho$ induces a natural saddle-point formulation for the projection problem, based on a pressure field $p$: Find $(\uu,p)\in L^2(\OO) ^2 \times H^1_\rho(\OO)$, where $H^1_\rho(\OO) = \{ q \in H^1_+(\OO), \, q = 0 \textmd{ a.e. on } [\rho<1]\}$, such that
\BE
\label{eq:sppmac}
\left \{
\begin{array}{lcl}
\uu + \nabla p &=& \UU \vseq \\
\displaystyle - \int_\OO \uu \cdot \nabla q & \leq & 0 \quad \hbox{ for all } q \in H^1_\rho(\OO),
\end{array}
\right . 
\EE
with the complementarity condition: 
$$- \int_\OO \uu \cdot \nabla p = 0 \, .$$



\section{Rigorous formalisms, well-posedness issues}
\label{sec:wp}

\subsection{Microscopic model~: generalized sweeping processes}
We address here the question of existence and uniqueness of solutions to 
$$
\frac {d\qq}{dt } = P_{C_\qq} \UU(\qq).
$$
where $t \mapsto \qq(t)\in  \setR^{2N}$ describes the motion of the crowd, $C_\qq$ is the cone of feasible velocities defined  by~(\ref{eq:Cqdef}), and $\UU(\qq)$ is the desired velocity. 
%
Before stating the main results of this section, let us describe the general framework into which it fits, namely the 
 sweeping processes, introduced by Moreau in the late 70's (see~\cite{Mor}). The original setting was the following: consider $V$ a Hilbert space and 
$t\longmapsto K(t) \subset V$ a path  of convex closed sets in $V$, with some regularity in time (e.g.\ $K(t)$ is continuous for the Hausdorff distance). He was interested in the evolution of a point $q(t)$ subjected to remain in $K(t)$. Assuming that the evolution process tends to minimize the norm of $\dot q$, one ends up with the following process
$$
\frac {d q}{dt } \in - N_{K(t)} (q(t)),
$$
where $N_{K}(q)$ is the outward normal cone to $K$ at $q$:
\BE
\label{eq:onc}
N_{K}(q) = \set{  v \in V \virg \exists \alpha >0\virg q \in P_K(q + \alpha v).
}
\EE
This normal cone has been introduced in \cite{Clarke} and is called more precisely the proximal normal cone. 
Note that,  as $K$ is convex, this normal cone identifies to $\partial I _{K}$, the subdifferential of  the  indicatrix function of $K(t)$. 
The following discrete process (so called {\em catching up} algorithm), which is used by Moreau  to establish well-posedness of the problem, gives a clear idea of the evolution mechanism. Consider a time step $\tau >0$, an initial condition $q^0\in K(0)$, successive positions are built according to 
$$
q^{n} = P_{K(t^n)}(q^{n-1}).
$$ 
As $K(t^n)$ is closed and convex, the projection is well-defined and contractant, which allows to establish convergence results for the sequence of  piecewise constant solutions $q_\tau$. 
It is clear that the recursive process extends straightforwardly to the case where the projection onto $K$ is properly defined (i.e. single-valued) into its neighborhood.  The finite-dimensional sets satisfying this property were introduced by Federer in \cite{Federer} under the name of {\em positively reached sets}. Then, they were called {\em p-convex sets}  by Canino in \cite{Canino} and later {\em proximally smooth sets} by Clarke, Stern and Wolenski in \cite{Clarke}. The final name ``uniformly prox-regular set'' will be given by Rockafellar et al. in \cite{Poliquin1,Poliquin2} (see Definition~\ref{def:pr} below).

%
%
%
%
%
%

The crowd motion model differs slightly from the sweeping process, as the feasible set is fixed whereas the point $\qq$ (which represents the whole crowd) tends to evolve according  to some given velocity $\UU$. The same principles can be extended straightforwardly, in particular: assuming $\UU$ is not too large, and $K$ is uniformly prox-regular, then for a sufficiently small time step $\tau$, the catching up algorithm can be used to build discrete solutions. In the same manner, 
one may 
write the evolution process as 
\BE
\label{eq:dqNmic}
\frac {d \qq}{dt }  - N_{K} (\qq(t)) \ni \UU(\qq(t)).
\EE
Note that this new formulation is a direct consequence of~(\ref{eq:pbmic}): as $N_{K}(\qq)$ and $C_\qq$ are mutually polar, the identity operator  can be written as $I = P_{N_{K}(\qq)} + P_{C_\qq}$ (see~\cite{moreauPol}).
Equivalence between both formulations is not obvious, as an identity has been replaced by an inclusion, yet we shall see that it holds true under some conditions.


The catching up algorithm reads as follows 
$$
\hbox{Step 1 (prediction):  }  \tilde \qq ^{n+1} = \qq^n + \tau \UU(\qq^n);
$$
$$
\hbox{Step 2 (correction):  } \qq ^{n+1} = P_K \left ( \tilde \qq ^{n+1}\right ) .
$$

Let us now express in a rigorous manner that this is indeed an algorithm for $\tau$ sufficiently small, and that this strategy allows to build a solution to our problem. 

 Let us first  give a proper definition of prox-regularity (see Theorem 1.3 in \cite{Poliquin2}). 
\begin{definition}
\label{def:pr}
A closed set $K$ is said to be $\eta$-prox-regular  at $\tilde{q}$ if  there exists a neighborhood $O$ of $\tilde{q}$ such that
for all $q \in \partial K \cap O $ and $v \in N_K(q)$, with $\abs {v} =1 $,  we have
$$B(q+\eta v, \eta)\cap K = \emptyset,$$
where $ N_K(q)$ is the proximal normal cone to $K$ at $q$ (defined by~(\ref{eq:onc})).
It is said to be uniformly prox-regular with a constant $ \eta$ if it is $\eta$ prox-regular at every point of its boundary. 
\end{definition}
As for  our microscopic model, it holds
\begin{proposition}
K is uniformly prox-regular with 
 $$ \eta \leq r\sqrt{ \dfrac{12}{N (N-1)(N+1)}}.$$
\end{proposition}
Proof. 
First it can be shown that $K_{ij}=\{ \qq \in \setR^{2N}, D_{ij} \geq 0 \} $ is uniformly prox-regular. In this case, tools of differential geometry can be used to compute the constant of prox-regularity which corresponds to the smallest
radius of curvature (i.e. the largest eigenvalue of
Weingarten operator). After calculation, we show that $K_{ij} $ is uniformly
prox-regular with constant $r\sqrt{2} $. 

One may wonder whether the intersection of such sets (which is the case for $K$) is uniformly prox-regular with a constant depending only on the constants of prox-regularity 
of the smooth sets. From a general point of view, this is wrong as illustrated in Figure~\ref{fig:idfausse}. 
Indeed, we have plotted in solid line the boundary of a set $S$ which is the intersection of two identical disks' complements. 
This set is uniformly prox-regular but its constant of prox-regularity (equal to the radius of the disk plotted in dashed line) tends to zero
 when the disks' centers move away from each other. In this situation, the scalar product between the normal vectors $\nn_1$ and $\nn_2$ 
(see Fig.~\ref{fig:psvectnorm}) tends to -1, which suggests that the prox-regularity constant also depends on the angle between normal vectors.

\begin{figure}[h]
\begin{center}
\includegraphics[width=0.3\textwidth]{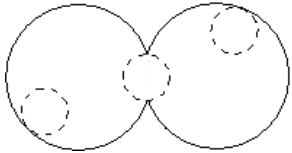}
\caption{Vanishing of the constant of prox-regularity.}
\label{fig:idfausse}
\end{center}
\end{figure}
\begin{figure}[h]
\centering
\psfrag{n1}[l]{$\nn_1$}
\psfrag{n2}[l]{$\nn_2$}
\includegraphics[width=0.2\textwidth]{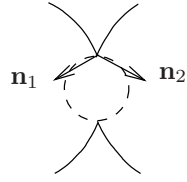}
\caption{Evolution of the angle between the vectors $\nn_1$ and $\nn_2$.}
\label{fig:psvectnorm}
\end{figure}

Consequently the proof of the uniform prox-regularity of $K$ rests on a good estimate of the angles between gradients of active constraints.  The existence of such a constant $\eta $ relies on the positive linearly independence of gradients of active constraints. 
More precisely it can be shown that there exists $\gamma >1$ such that for all $\qq \in K $, 
\BE
\label{eq:revtriang}
 \sum_{(i,j) \in \Lambda(\qq)}  \lambda_{ij} |\GG_{ij}(\qq) | \leq \gamma  \left| \sum_{(i,j) \in \Lambda(\qq)}  \lambda_{ij} \GG_{ij}(\qq)\right|. 
\EE
It can be  deduced from this inequality that $K$ is $\eta $-prox-regular with $\eta = r\sqrt{2}/ \gamma$ (a detailed proof of this result can be found in ~\cite{MVm2an, Numschemejv}).
Furthermore, the given upper bound is obtained by considering a configuration $\qq$ of aligned disks and by computing the constant of local prox-regularity at $\qq$ (see Proposition 3.18 in ~\cite{thesejv} for more details).
\endproof

Note that the prox-regularity coefficient of $K$ degenerates
(i.e. goes to $0$) as the size of the disks goes to zero, for a
constant total mass. Furthermore the prox-regularity coefficient of
$K$ depends on $N$ with $N^{3/2}$ for $N$ large enough (see \cite{Saisho}).

Now the following result can be established:
\begin{theorem}
Let $\UU$ be Lipschitz and bounded, for all $T>0$ and all $\qq_0 \in \ K $,
there exists a unique absolutely continuous solution $t \longmapsto \qq(t)$ to $$\left \{ 
\begin{array}{l}
\displaystyle\frac {d \qq}{dt} + N_K (\qq) \ni \UU(\qq) \  a.e. \ in \ [0,T] \vspace{6pt}
\\
\qq(0)=\qq_0. 
\end{array}
\right.$$
Moreover, this solution satisfies the differential equation
 $$\left \{ 
\begin{array}{l}
\displaystyle\frac {d \qq}{dt} = P_{C_\qq}(\UU(\qq) )\  a.e. \ in \ [0,T] \vspace{6pt}
\\
\qq(0)=\qq_0. 
\end{array}
\right.$$
\label{th:wpmicro}
\end{theorem}
The well-posedness of the differential inclusion can be proved
thanks to results in \cite{Thi1,Thi2}. Then Proposition
3.3 in \cite{Proxbanach} claims that the solution satisfies the
following differential equation $$\frac {d \qq}{dt} +
P_{N_K(\qq)}(\UU(\qq) )=\UU(\qq) $$ which is equivalent to
(\ref{eq:pbmic}) since the cones $ N_K(\qq)$ and  $C_\qq$ are mutually
polar.

\subsection{Macroscopic model}
\label{sec:theomac}
We are concerned here with the question of existence of solution to the macroscopic model: given a desired  velocity field $x\mapsto \UU(x)$ defined in $\OO$, given an initial density $\rho^0$, find $\rho(x,t)$ such that 
\BE
\label{eq:pbmactransp}
\partial_t \rho  + \nabla \cdot (\rho \uu) = 0,
\EE
where  the actual velocity field $\uu$ verifies
\BE
\label{eq:pbmacproj}
\uu = P_{C_\rho}\UU,
\EE
and $C_\rho$ is the set of feasible velocities
defined by \eqref{eq:defCrho}.
Eq.~\ref{eq:pbmactransp} is meant in a weak sense:
$$
\int_0^T \int_\OO \rho \partial_t  \varphi  + \int_0^T \int_\OO \rho \uu \cdot \nabla \varphi 
+ \int _\OO \varphi(0,\cdot) \rho^0 = 0
\quad \forall \varphi \in C_c^\infty([0,T)\times \setR^2).
$$

 The overall problem can be written as a non-local and nonlinear transport equation
$$
\partial_t \rho  + \nabla \cdot (\rho \uu (\rho))= 0,
$$
where the notation $\uu(\rho)$ expresses a dependence of $\uu$ upon the whole density field $\rho$, through the projection of $\UU$ onto $C_\rho$. The nonsmooth character of this projection and the fact that the regularity of $\uu$ is not controlled (the projection is performed in the $L^2$ sense), rule out the possibility to use standard tools to define solutions to this equation. 
One may wonder whether the catching-up approach, which proved successfull  in the microscopic setting, can be followed in the present context. The prediction step is straightforward: one transports $\rho$ according to the desired velocity field $\UU$ during a time step $\tau >0$. 
 As for the correction (or projection) step, it appears immediately that the standard eulerian way to measure distances between densities (by estimating a given norm of  their difference) is not suitable. In the microscopic setting, the catching up approach proved successful because the difference between two configurations identifies with a displacement field, and this is due to the very Lagrangian description of individuals. Recovering this feature calls for a new way to measure differences between densities, more respectful of the Lagrangian  description of the motion, and this way is the Wasserstein setting. 
 
 \bigskip
 
 \para{Optimal transportation and Wasserstein distance}
We assume here that $\OO$ is convex.
 Let $\bft \;: \; \OO \longrightarrow \OO$ be a measurable mapping, and $\mu\in \PP(\OO)$  a probability measure. We say that 
 $\bft $ pushes forward $\mu$ onto  $\nu$ (written $\bft_\# \mu = \nu$) if
 $$
 \mu\left ( \bft^{-1} (A)\right )  = \nu (A)
 $$
 for any measurable set $A$.
 For given measures $\mu_0$ and $\mu_1$, we denote by $\Pi(\mu_0,\mu_1)$ the set of transport maps between $\mu_0$ and $\mu_1
 $. The quadratic Wasserstein distance  $W_2(\mu_0,\mu_1)$ is defined by
 $$
 W_2(\mu_0,\mu_1)^2 =  \inf_{\bft\in \Pi(\mu_0,\mu_1) }\int_\OO \abs {\bft(x) - x} ^2 d \mu_0(x) 
 .
 $$
 
 Notice that this is a sloppy definition that holds for atomless measures, which is anyway the case we are interested in; for general measures one should pass through the so-called {\it transport plans}, which we do not want to introduce here, and we address the interested reader to \cite{Amb,Vil1}. However, this quantity $W_2$ can be proven to be a distance on the space $\PP(\OO)$ and it makes the space of probability measures a {\it geodesic space}, i.e. every pair of points is linked by a curve (which is in this case a curve of measures), such that the length of this curve exactly equals the distance between the points. In case in the minimization defining $W_2$ there is an optimal transport map $\bft$ (which is the case if $\mu_0$ is absolutely continuous with respect to the Lebesgue measure $\lcal^d$), then this geodesic curve is known explicitly and it is given by $\mu_t:=((1-t)\mathbf{id}+t\bft)_\#\mu_0$.
 
 In the spirit of the structure that we are imposing on $\PP$, we define
the  outward normal cone to a subset $K\subset \PP$ (we actually  apply the definition of the strong Fr\'echet subdifferential in~\cite{Amb} to  the indicatrix function of $K$) by
$$
\vv \in N_K(\rho) \Longleftrightarrow  \int_\OO \vv(x) \cdot (\bft(x)-x)d\rho(x) \leq o\left (\| \bft-\mathbf{id}\|_{L^2(\rho)} \right ) 
$$
for all $\bft$ such that $\bft_\# \rho \in K$. It makes it possible to express the problem like we did in the microscopic setting (Eq.~(\ref{eq:dqNmic})):
\BE
\label{eq:dqNmac}
\uu - N_K (\rho) \ni \UU,
\EE
and $\uu$ transports $\rho$ according to~(\ref{eq:pbmactransp}).


We may now write the catching up algorithm (which we consider for the time being as a theoretical tool): given an initial density $\rho^0$ and a time step $\tau > 0$, build $\rho^1$, \dots, $\rho^n$ according to 
\BE
\label{eq:cumac}
\left \{
\begin{array}{rcll}
  \tilde \rho^{n+1} &= &\left ( \mathbf{id} + \tau \UU \right ) _\# \rho ^n &\hbox{  transport (prediction), }\vseq\\
 \rho^{n+1} &=& P_K \left ( \tilde \rho ^{n+1}\right ) &\hbox{  projection (correction),  }
\end{array}
\right . 
\EE
where the projection is performed in the Wasserstein sense and the set $K$ is the constrained set introduced in \eqref{definition K}. As for the microscopic model, we must check that both steps are well defined (possibly under some restriction on the time step $\tau$), and that the obtained discrete solutions converge to a limit which can be identified as a solution to problem~(\ref{eq:pbmactransp})(\ref{eq:pbmacproj}).

We shall assume here two properties of $\UU$. First we require that it tends to keep people within the room $\OO$, i.e. that for $\tau $ sufficiently small $ \left ( \mathbf{id} + \tau \UU \right ) $ maps $\OO$ onto $\OO$. Obviously, in the case of an emergency evacuation, with an open exit door, we will of course alleviate this assumption by allowing $\UU$ to cross the exit line (we will discuss later on how to ``catch up'' the possible part of the mass that exits $\OO$ because of this). Secondly, we need it regular enough (Lipschitz continuous is sufficient), so that, still for $\tau $ sufficiently small $ \left ( \mathbf{id} + \tau \UU \right ) $ preserves the absolute continuity of the measure, i.e. $\rho\!<\!\!<\!\lcal^d$ implies $ \left ( \mathbf{id} + \tau \UU \right )_\#\rho\!<\!\!<\!\lcal^d $.
This assumption being made, the prediction step is well defined for any measurable field $\UU$.

The key point is the projection step, which differs significantly from the microscopic situation. It is a variational problem in the space of measures, which consists in minimizing the distance to a fixed measure among elements of $K$.
As for existence, standard compactness arguments may be applied, but the question of uniqueness is more delicate. 

 First of all, let us rule out the possibility to establish uniform prox-regularity in the spirit of Definition~\ref{def:pr}. Consider the one-dimensional situation. A Dirac mass at zero projects onto the characteristic function of $(-1/2,1/2)$. More generally, a combination of Dirac masses 
 $$
 \mu = \sum \alpha_n \delta _{x_n}\virg  \sum \alpha_n = 1,
 $$
 (we assume that the distance between  supports of any two of them is always larger than half the total mass carried by the couple, so that they do not interact) projects onto the sum of characteristic function of the flattened Dirac masses:
 $$
  P_K (\mu) =  \sum \ONE _{(x_n-\alpha_n/2,x_n+\alpha_n/2)}.
  $$
  The Wasserstein  cost can be computed straightforwardly, it is 
  $$
  C =\frac 1 {12} \sum \alpha_n ^3.
  $$
On the other way around,  consider a density $\rho \in K$ (i.e.\ $\rho(x) \leq 1$ almost everywhere) and $\omega$ the largest open set such that $\rho(x)  =  1$  a.e.\ in $\omega$. This open set $\omega$ is a countable union of open intervals of lengths $(\alpha_n)_{1\leq n < N}$ (with possibly $N= +\infty$). Among all densities which projects onto $K$  at $\rho$, the farthest is the combination of  Dirac masses with weights $\alpha_n$, supported at the middles of intervals, and it is at distance $(\sum \alpha_n^3/12)^{1/2}$. As a consequence, $\eta$ prox-regularity can be expected at some points, but $\eta$ is not bounded away from $0$. Note that some measures are the projection of themselves only, even if they saturate the constraint at every point of their support: consider e.g. a dense open set of total measure $1$ in $(-1,1)$, and define $\rho$ as the characteristic function of its complement. As the saturated zone contains no interval, it cannot be the projection  of a 
 measure which violates the constraint. 
 
Note that this non prox-regularity can be seen as  a natural consequence of the fact that the property in the microscopic setting degenerates as the granularity gets finer (i.e.  when the radius goes to $0$).

On the other hand, the set $K$ enjoys some kind of convexity property: it is geodesically convex (see \cite{Amb}) : 

\begin{definition}
$K \in \PP(\OO)$ is said to be geodesically convex if for any $\mu_0$, $\mu_1$ $\in K$, the geodesic curve $\mu_t$ joining $\mu_0$ and $\mu_1$ belongs to $K$ for all $t \in [0,1]$.
\end{definition}

This property might suggest that uniqueness of the projection can be obtained by standard arguments: in a so-called  Aleksandrov non-positively curved (NPC) length space (see \cite{Ale}),  considering $2$ minimizers and the measure halfway along the geodesic, convexity properties of the distance function $\mu\mapsto W_2(\rho,\mu)^2$ lead to a contradiction. 

 Unfortunately, as soon as  $d \geq 2$,  $\PP(\OO)$ is not NPC and the previous distance function even turns out to have {\it concavity }properties.
 
 Yet, despite the previous assertions (non prox-regularity  and the concavity of the squared distance), other kind of interpolation curves (called {\it generalized geodesics}) can be used to assert well-posedness of the projection problem, without any restriction on the distance from $K$.
 \begin{definition} (Generalized geodesics, see~\cite{Amb}, p. 207)
 A generalised geodesic joining $\mu_0$ to $\mu_1$ with base $\rho$ is a curve defined by
 $$\mu_t = ((1-t) \mathbf{r}_0 + t \mathbf{r}_1)_\# \rho,$$
 where $\mathbf{r}_0$ (resp. $\mathbf{r}_1$) is an optimal transport map from $\rho$ to $\mu_0$ (resp. $\mu_1$).
 \end{definition}
 
Notice that when $\mathbf{r}_0= \mathbf{id}$ this gives again the expression of a geodesic in $\PP$ according to the distance $W_2$. 

This generalized notion can be used to establish the following proposition
 \begin{proposition}
 For any $\rho\in \PP(\OO)$,  the Wasserstein distance to the set $K$ of admissible densities defined in \eqref{definition K}
 is attained at a unique point $P_K \rho$. 
 The projection operator $P_K$  is continuous. 
 \end{proposition}
 
\begin{proof}
It is easy to check (but one can refer to Lemma 9.2.1 p.206 in~\cite{Amb}), that the square Wasserstein distance enjoys strict convexity properties along generalized geodesics, i.e. $W_2(\rho,\mu_t)^2<(1-t)W_2(\rho,\mu_0)^2+tW_2(\rho,\mu_1)^2$. This gives the uniqueness once we know that the set $K$ is still convex along these generalized geodesics, i.e. that $\mu_0, \mu_1\in K$ implies $\mu_t\in K$. This is obtained by a computation of the Jacobian factor of the map $(1-t) \mathbf{r}_0 + t \mathbf{r}_1$, see \cite{MRS10}.
\end{proof}

Now that the catching-up algorithm is  defined properly, the obtained sequence of discrete solutions allows to obtain an existence result:

\begin{theorem} Let $\rho_{\tau}$ be the piecewise constant interpolation of the discrete densities given by the catching-up algorithm~(\ref{eq:cumac}). 
If $\mathbf{U}$ is a ${C}^1$ velocity field, and $\rho^0 \in K$, then $\rho_\tau$ converges as $\tau$ tends to $0$ to a solution of the macroscopic problem
$$\left\{ \begin{array}{rcl}
\partial_t \rho + \nabla \cdot (\rho \mathbf{u}) & = & 0,\\
\mathbf{u} & = & P_{C_\rho} \mathbf{U}.
\end{array} \right.$$
\end{theorem}

\begin{proof}
Let us first describe the discrete quantities we obtain thanks to the catching up algorithm. We denote by $\mathbf{r}^{n+1}$ the optimal transport between the projected density $\rho^{n+1} = P_K \tilde \rho^{n+1}$ and $\tilde \rho^{n+1}$, and we write $\mathbf{t}^{n+1} =  (\mathbf{id} + \tau \mathbf{U})^{-1}$ (see figure \ref{fig:transports}). These transport maps allow to define a discrete velocity as follows:
$$\mathbf{v}^{n+1} = \frac{\mathbf{id} - \mathbf{t}^{n+1} \circ \mathbf{r}^{n+1}}{\tau}.$$

\begin{figure}[h!]
\begin{center}
\psfrag{t}[l]{$\mathbf{t^{n+1}} = (\mathbf{id} + \tau \mathbf{U})^{-1}$}
\psfrag{r}[l]{$\mathbf{r}^{n+1}$}
\psfrag{rho}[l]{$\rho^{n}$}
\psfrag{trho}[l]{$\tilde \rho^{n+1}$}
\psfrag{rho1}[l]{$\rho^{n+1}$}
\psfrag{K}[l]{$K = [\rho \leq 1]$}
\includegraphics[width=0.5\linewidth]{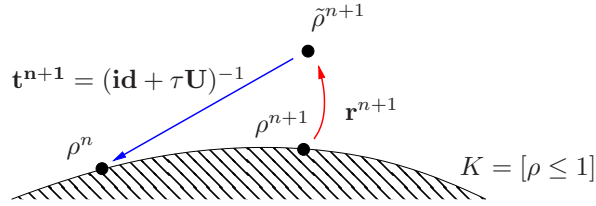}
\caption{Definition of the discrete transport maps}
\label{fig:transports}
\end{center}
\end{figure}

We then define two different interpolations of these quantities. First, the piecewise constant interpolation is given by :
$$\left\{ \begin{array}{rcl} \rho_\tau(t,.) & = & \rho^{n+1}\\
\mathbf{v}_\tau (t,.) &  = & \mathbf{v}^{n+1}
\end{array} \right. \quad \textmd{ if } \; t \in ]n \tau, (n+1) \tau].$$
Using \cite{MRS10} for the projection part, it is possible to prove that $\mathbf{v}_\tau$ satisfies the following discrete decomposition
$$\mathbf{U} = \mathbf{v}_\tau + \nabla p_\tau + \varepsilon_\tau,$$
where $\varepsilon_\tau$ converges uniformly to $0$ as $\tau$ tends to $0$.
We also define a continuous interpolation of $(\rho^n)_n$ as follows
$$\tilde \rho_\tau (t,.) = (\mathbf{T}^{n+1}_t)_\# \rho^{n+1}, \quad  \textmd{ if } \; t \in ]n \tau, (n+1) \tau],$$
where $\mathbf{T}^{n+1}_t = \left( (t - n \tau) \mathbf{v}^{n+1} + \mathbf{t}^{n+1} \circ \mathbf{r}^{n+1} \right)$. This second interpolation satisfies the transport equation at velocity $\mathbf{\tilde v}^{n+1} = \mathbf{v}^{n+1} \circ (\mathbf{T}^{n+1}_t)^{-1}$.

It is possible to prove a priori estimates on these interpolated curves, and therefore prove that they both converge to the same limit. The continuous interpolation of $(\rho^n)_n$ gives that the limit density satisfies a transport equation, and the discrete decomposition of the piecewise constant interpolation proves that the limit velocity is indeed the projection of the desired velocity onto $C_\rho$.
\end{proof}

We finish this section by considering the case where, due to modeling  reasons, we allow the vector field $\UU$ to let the mass exit through a part of the boundary. This means that $\mathbf{id}+\tau\UU$ is no longer supposed to map $\OO$ into $\OO$ but the segment connecting $x$ to $x+\tau\UU(x)$ is allowed to cross $\partial\OO$ in a prescribed subset $\Gamma\subset\partial\OO$ which stands for the exit. In such a case the transport step of the algorithm does not need to be changed, but we have to face, during the projection step, a density $\tilde\rho^{n+1}$ which is no longer supported in $\overline\OO$. To bring back the mass to the original domain we want to consider a modified set $K$, that we call $K_\Gamma$:
\begin{equation}\label{definition KGamma}
K_\Gamma = \set{ \rho\in\PP(\OO)\,:\,\rho=\rho_\OO+\rho_\Gamma\virg \supp (\rho_\Gamma) \subset \Gamma\virg 0\leq \rho_\OO (x) \leq 1\hbox{ for a.e.  } x}.
\end{equation}

The reason for this choice is the following: we consider that as soon as a particle reaches $\Gamma$, instead of following its movement after $\Gamma$, we leave it on it. This is done for simplicity, but it only means that we are no longer concerned with what happens to the particles that have reached $\Gamma$, not that they are really blocked on the exit. Obviously, we needed to withdraw the density constraint on $\Gamma$, so as to let particles stay on it, and also to represent the fact that $\Gamma$ stands actually for everything that happens at the door {\it and beyond}.

Once the set $K_\Gamma$ is introduced, the projection step is done by defining 
$$ \rho^{n+1} = P_{K_\Gamma} \left ( \tilde \rho ^{n+1}\right ),$$
which means that we take a measure whose support may go beyond $\Omega$ and we project it onto the measures over $\OO$ satisfying some extra density constraint. We stress that the same kind of argument (projecting on a set of measures concentrated on $\Omega$) could also be used, in the case with no exit, to withdraw the (yet, natural) assumption that $\mathbf{id}+\tau\UU$ maps $\OO$ into $\OO$.

The mathematical problem with this modified $K_\Gamma$ is much trickier. One of the difficulties, that prevent the usual theory to be applied, is the fact that the set $K_\Gamma$ loses some of the properties that $K$ had previously (in particular it is no more  geodesically convex: the geodesic - for the $W_2$ distance - between two points of $K_\Gamma$ could go out of  $K_\Gamma$). In particular we have no clue about the uniqueness of the projection, but the algorithm works in the same way if we accept to take any minimizer of the distance. 

\subsection{Gradient flow setting}
In case the spontaneous velocity  is the gradient of some dissatisfaction function (e.g.\ distance to the exit in case of an emergency evacuation), the microscopic model can be put into a gradient flow form:
$$
\frac {d\qq}{dt } \in  - \partial \varphi \virg \varphi = \left ( \Psi + I_K\right ) ,
$$
where $\partial \varphi$  is the Fr\'echet subdifferential of $\varphi$ defined by 
$$
\partial \varphi(\qq) = \set{ \vv \in \setR^{2N}\virg  \varphi(\tilde{\qq} )-\varphi(\qq)-\vv \cdot (\tilde{\qq} -\qq) \geq o \ (|\tilde{\qq} -\qq|)
}.
$$
The function $\Psi$ that we consider is the total dissatisfaction, i.e. $\Psi(q)=\sum_i D(q_i)$, where $D(x)$ may be, as we said, the distance to the exit. 
Indeed, as $K$ is uniformly prox-regular, the proximal normal cone identifies with the Fr\'echet subdifferential of $I_K$ hence $\partial \varphi = \nabla \Psi + N_K(\qq)  $ (see \cite{Rockafellar} for more details).

In the macroscopic case, under the same assumption that $\UU$ is a gradient $-\nabla D$, then one defines the global dissatisfaction function in a continuous setting:
$$
\Psi(\rho)  = \int D(x) \, \rho(x) \, dx,
$$
and the overall process can be seen as a gradient flow in the Wasserstein space, i.e. $\rho$ is advected by $\uu$, 
$$
\partial_t \rho  + \nabla \cdot (\rho \uu) = 0,
$$
with 
$$
\uu \in -\partial \varphi (\rho) \hbox{ for a.e. } t,
$$
where $\varphi = \Psi + I_K$, and $\partial \varphi $ is defined in the following sense (see~\cite{Amb}):
%
\begin{definition}
The strong Fr\'echet subdifferential of a function $\varphi$ at $\rho$ is the set of fields $\mathbf{u}$ such that for all transport maps $\mathbf{t}$, the following inequality holds true
$$\varphi(\rho) + \int_{\mathbb{R}^d} 
\mathbf{u}(x)\cdot(
 \mathbf{t}(x) - x 
)
 d\rho(x) \; \leq \; \varphi(\mathbf{t}_\# \rho) + o(||\mathbf{t}-\mathbf{id}||_{L^2(\rho)}) .$$
Notice that the definition we gave before of $N_K$ is nothing but a particular case of this one, when $\varphi=I_K$.
\end{definition}

The main advantage of this gradient-flow formulation is the fact that it allows to handle vector fields $\UU$ less regular than what we need in the general case (where it is natural to require $\UU$ to be Lipschitz continuous, i.e. $D\in C^{1,1}$).  For the whole theory on gradient flows, first in a finite-dimensional setting, then in Hilbert spaces, and finally in metric spaces and particularly in $\PP(\OO)$, we refer again to \cite{Amb}.

In particular, the catching up algorithms made by the coupling of a prediction and a correction step, may be replaced in the case of a gradient structure by a single-step procedure, called {\it proximal algorithm}, where
$$q^{n+1}\in \argmin \left\{ \varphi(q)+\frac {|q-q^n|^2}{2\tau} \right\}.$$
This algorithm has also been formalized in a metric setting, where it is known as {\it minimizing movements} (see \cite{DeG, MovMin}): it has first been applied to the case of $\PP(\OO)$ with the Wasserstein metric by Jordan, Kinderlehrer and Otto in \cite{Ott1}.

The details about a gradient flow approach to the macroscopic framework of crowd motion are contained in \cite{MRS10}, where both the case with no exit (i.e. with the constraint $\rho\in K$) and with exit ($\rho\in K_\Gamma$) are dealt with, the second one being much trickier than the first.

\section{Numerical solution} 

\subsection{Microscopic model}

Moreau's catching up algorithm suggests a strategy to build approximate solutions to the evolution problem. Yet, projecting a configuration $\qq$ onto $K$ is not straightforward. The scheme we propose extends ideas introduced in~\cite{maury06} for granular flows. It consists in performing a catching up step with   $K$  replaced by some kind of inner local convex approximation. More precisely, considering that the configuration $\qq^n$ at time $t^n$ is known, the predicted configuration is obtained by 
\begin{equation}
\label{eq:nummicropred}
\tilde \qq^{n+1} = \qq^n + \tau \UU(\qq^n).
\end{equation}
Then $\qq^{n+1}$ is obtained by projecting  $\tilde \qq^{n+1}$ onto $K_{\qq^n}$
\begin{equation}
\label{eq:nummicrocorr}
\qq^{n+1} = P_{K_{\qq^n}}  \tilde \qq^{n+1},
\end{equation}
 where $K_\qq$ stands for
$$
K_{\qq} = \set{ \tilde \qq\virg D_{ij}(\qq) + \GG_{ij}\cdot (\tilde \qq - \qq) \geq 0\quad \forall i\neq j }.
$$

\begin{center}
\begin{figure}
\centering
\begin{tabular}{c}
\psfrag{a}[l]{$\qq^n$}
\psfrag{b}[l]{$\qq^n +\tau \UU( \qq^n) $}
\psfrag{c}[l]{$\hspace{-7.5pt} \qq^n +\tau \UU( \qq^n)$}
\psfrag{d}[l]{$\qq^{n+1}$}
\psfrag{e}[l]{$\qq^{n+1}$}
\psfrag{f}[l]{$\hspace{-2.5pt}\bar{\qq}^{n+1}$}
\psfrag{g}[l]{$ \bar{\qq}^{n+1}$}
\psfrag{q}[l]{$ \Large  K $}
\psfrag{t}[l]{\hspace{-1cm}$ \Large{K_{\qq^n}}$}
\includegraphics[width=0.65\linewidth]{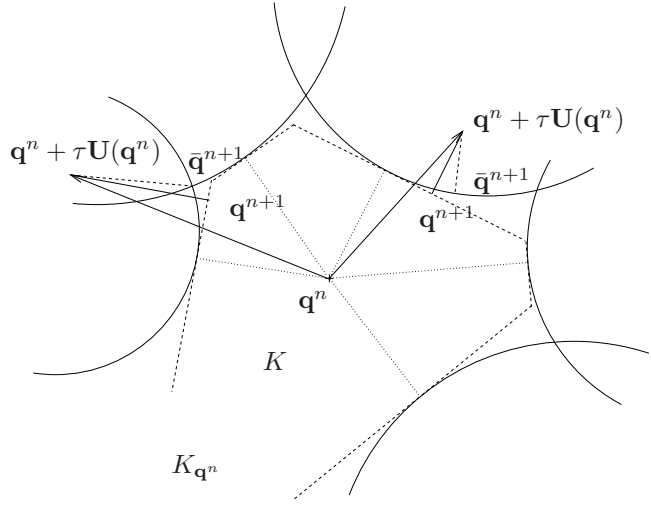} 
\end{tabular}
\caption{Theoretical and numerical projections.}
\label{fig:projs}
\end{figure}
\end{center} 
  \noindent In Figure~\ref{fig:projs},  we illustrate the set $K \subset \setR^{2N}$, intersection of sets $K_{ij}$ whose boundaries are plotted in solid line. The set $K_{\qq^n} $ is delimited by the dashed line. The theoretical and numerical projections, respectively $ \bar{\qq}^{n+1}:= P_{K} (\qq^n +\tau \UU( \qq^n)  )$ and $\qq ^{n+1}$ are represented (for two examples of $\UU( \qq^n)$). Indeed, as $K $ is uniformly prox-regular, the projection onto $K$ of $\qq^n +\tau U( \qq^n)$ is well-defined for $\tau $ small enough. 
  The replacement of $K $ by the convex set $ K_{\qq^n}$ is convenient because it allows us to use classical numerical methods to compute this projection. 
Indeed, this projection problem can be reformulated in a saddle-point form which can be solved by Uzawa algorithm : find 
$(\qq^{n+1}, {\mathbf \lambda}) \in \setR^{2N} \times (\setR^+)^{\frac{N(N-1)}{2}}$ satisfying
\BE\left\{
\begin{array}{l}
 \displaystyle \qq^{n+1}=  \tilde \qq^{n+1}+ \sum \lambda_{ij} \ \GG_{ij}(\qq^n)
\vseq \\
 \displaystyle \forall i<j,\ D_{ij}(\qq^n)+\GG_{ij}(\qq^n) \cdot
(\qq^{n+1}-\qq^{n} ) \geq 0
\vseq \\
\displaystyle \sum \lambda_{ij}\  ( D_{ij}(\qq^n)+\GG_{ij}(\qq^n) \cdot
(\qq^{n+1}- \qq^{n})=0.
\end{array}
\right.
\label{eq:projs}
\EE
Note that the Lagrange multipliers are 
not unique in
general. Indeed there exist configurations $\qq^n $ with strictly more
than $2N$ active constraints so that the associated gradients
$\GG_{ij}(\qq^n)$ are linearly dependent.  

  The matrix appearing in Uzawa algorithm is  $C=B\phantom{}^{t}{B } $ where $B $ is the matrix whose rows are vectors $\GG_{ij}(\qq)$. In \cite{MVm2an}, the authors quantify how the condition number of matrix $C$ varies with the parameter $\eta_\qq$ (setting a lower bound of the local prox-regularity of $K$ at point $\qq $) when $C$ is non-singular. 
As a consequence  we expect that the Uzawa algorithm converges less quickly for configurations with low local prox-regularity. In numerical simulations, we noticed indeed that solving the saddle-point problem requires more iterations in case of a jam.

\bigskip 

\para{Numerical analysis}
In the proposed scheme (\ref{eq:nummicrocorr}), the replacement of $K $ by the convex set $ K_{\qq^n}$ is computationally convenient
but arises many difficulties in the numerical analysis. The 
convergence result is based on the fact that $ K_{\qq}$ is a good
approximation of $K $ near $\qq$.

\begin{theorem} 
Let us denote by $\qq_\tau$ the continuous piecewise linear function satisfying $\qq_\tau(t^n)= \qq^n $ defined by (\ref{eq:nummicrocorr}). 
Let $\UU$  be Lipschitz and bounded, for all $T>0$, the sequence $(\qq_\tau)_\tau$ uniformly converges to $\qq$ defined in Theorem~\ref{th:wpmicro}, when $\tau $ tends to 0.  
\end{theorem}
\begin{proof}
Here we just give a sketch of the proof. We refer the reader to \cite{Numschemejv} for the details. 
It can be easily proved that $(\qq_\tau)_\tau $ is bounded and so a
convergent subsequence can be extracted.  Thanks to the uniqueness
result in Theorem \ref{th:wpmicro}, it suffices to check that the
limit function (which is absolutely continuous) satisfies the
differential inclusion.
Since $ \qq^{n+1}= P_{K_{\qq^n}} (\qq^n+ \tau \UU(\qq^n))$, it comes for all $\tilde \qq $,  
$$(\qq^n+ \tau \UU(\qq^n)-\qq^{n+1}) \cdot  (\tilde \qq - \qq^{n+1} ) \leq 
|\qq^n+ \tau \UU(\qq^n)-\qq^{n+1}|\,  d_{K_{\qq^n}}(\tilde \qq ). $$
By dividing by $\tau $, we obtain for all $ \tilde \qq $
$$(\UU(\qq^n)-\uu^n) \cdot  (\tilde \qq -\qq^{n+1} ) \leq 
|\UU(\qq^n)-\uu^n| \, d_{K_{\qq^n}}(\tilde \qq) \leq 2\|\UU\|_\infty \,  d_{K_{\qq^n}}(\tilde \qq) $$ where 
 $ \uu^n = (\qq^{n+1}-\qq^n)/ \tau $ represents the discrete actual velocity. 
We aim at passing to the limit in the previous inequality and the crucial term is the last one. 
Thus the convergence result rests on the local continuity of the map $(\qq, \qq_0)  \longrightarrow d_{K_{\qq_0}}(\qq) $
in a neighborhood of the set $\{(\qq,\qq_0), \ \qq=\qq_0\}$. More
precisely it can be checked that for all $\qq \in K$ and all $ \tilde
\qq \in B(\qq, r_\qq)$, $$P_{K_{\qq^n}}( \tilde \qq)  \xrightarrow[n
  \to  +\infty]{} P_{K_{\qq}}( \tilde \qq) $$ by using the saddle
point form of these projections (\ref{eq:projs}). Thanks to the
reverse triangle inequality, the Lagrange multipliers (which are not
unique in general) are bounded. Then compactness arguments allow us to
obtain the convergence of the projections. We deduce that the limit function
$\qq$ satisfies the following differential inclusion: $$
\dfrac{d\qq}{dt} + N(K_\qq, \qq) \ni \UU(\qq) \quad \textmd{ a.e. in } \, [0,T]. $$
As for all $\qq \in Q $, $ N(K_\qq, \qq) =N(K, \qq)$,  the required
result is proved. 
\end{proof}

The convergence order of this scheme will be specified in a
forthcoming paper \cite{convorder}. This more accurate result is based on metric qualification
conditions between sets $ K^{ij}_\qq=\set{ \tilde \qq\virg D_{ij}(\qq) +
  \GG_{ij}\cdot (\tilde \qq - \qq) \geq 0 }$. More precisely there
exists a constant $ \theta >0 $ such for all $\qq \in K$ and all $ \tilde
\qq \in B(\qq, r_\qq)$,  $$ d_{K_\qq} ( \tilde \qq)\leq  \theta \sum
d_{K^{ij}_\qq} ( \tilde \qq)   .$$ This result rests on the reverse
triangle inequality~(\ref{eq:revtriang}).

\subsection{Macroscopic model}

The strategy we propose relies again  on the catching up philosophy which already made it possible to establish an existence result. The density is first advected by the spontaneous velocity field (with no consideration of the congestion constraint), and then projected (in the Wasserstein sense) onto the set of feasible densities. The time discretization scheme writes as follows: Given a time step $\tau > 0$, an initial configuration $\rho^0$, approximate configurations $\rho^1$, \dots, $\rho^n$ are built recursively according to 

\begin{equation}
\label{eq:nummacropred}
\tilde \rho^{n+1} =\left ( \mathbf{id} + \tau \UU\right ) _\# \rho^n  \quad \hbox{(prediction step)},
\end{equation}
\begin{equation}
\label{eq:nummacrocorr}
\rho^{n+1} = P_{K}  \tilde \rho^{n+1}  \quad \hbox{(correction step)}.
\end{equation}
The first step consists in transporting $\rho^n$ according to the given transport map $\mathbf{id} + \tau \UU$. 
%
We simply transport the center of each cell at velocity $\mathbf{U}$ during a time $\tau$, and then distribute the mass of the transported cell on the different cells of the mesh it intersects, as illustrated in Fig.~\ref{fig:lagrangien}. The transported density writes 
$$\tilde \rho^{n+1} = \frac{1}{dx dy} \sum_{i,j} \sum_{k,l} \rho^n_{i,j} \; \mathcal{A}\left(\textmd{ transported cell (i,j) } \cap \textmd{ cell (k,l) } \right),$$
where $\mathcal{A}(B)$ represents the area of the set $B$.\\
\begin{figure}[h!]
\begin{center}
\psfrag{c}[l]{$c_{i,j}$}
\psfrag{cp}[l]{$c'_{i,j}$}
\psfrag{U}[l]{$\mathbf{id} + \tau \mathbf{U}$}
\includegraphics[width=0.5\linewidth]{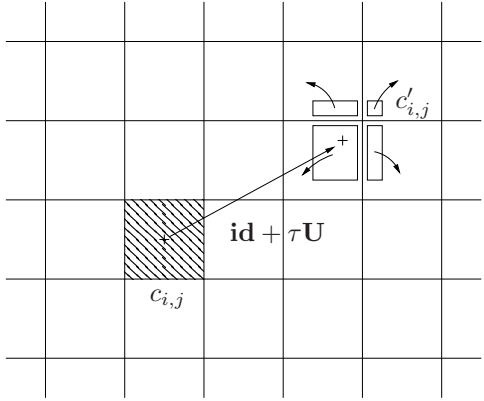}
\caption{Lagrangian transport of the density}
\label{fig:lagrangien}
\end{center}
\end{figure}
%
The second step is less standard. 
The scheme we propose to approximate the 
 projection onto $K$ with respect to the Wasserstein distance is based on the following considerations\footnotemark
 \footnotetext{To alleviate the presentation, we do not normalize measures to a unit total mass, but the considerations on optimal transportation and Wasserstein distance for probability measures which we presented at the beginning  of Section~\ref{sec:theomac} are straightforwardly extended to measures with arbitrary  mass.}: Consider a domain $\omega$ and a non-zero density $\mu$ supported by $\overline\omega$. For $\eps > 0 $, $1_\omega + \eps \mu \notin K$, and its projection  onto $K$ is identically $1$ in $\omega$. Let us denote by $\eps\uu$ the displacement field which corresponds to the optimal transport between $1_\omega + \eps \mu$ and $P_K(1_\omega + \eps \mu)$. One has
 $$
 \left ( \mathbf{id} + \eps\uu \right ) _\#  \left ( 1_\omega + \eps \mu \right )  = P_K(1_\omega + \eps \mu).
 $$
 As the latter is identically $1$ in $\omega$, one has
 $$
 \frac {1 + \eps \mu}{det (I +\eps\nabla  \uu )} = 1,
 $$
 so that, at the first order in $\eps$, 
 $$
 \nabla \cdot \eps\uu = \eps \mu.
 $$
 As $\mathbf{id} + \eps\uu $ is an optimal map, $\uu$ is a gradient~: $\uu = -\nabla p$, where $p$ verifies the Poisson problem
 $$
 -\Delta p =  \mu .
 $$
 Since we are considering a violation of the density constraint on a set $\omega$ and assuming that, after the projection, the density will be saturated on $\omega$ itself, we are only interested in finding the mass that will exit $\omega$ when the displacement is given by $\eps\uu$. This means, at a first order approximation, we only need to estimate the flux of density through the boundary $\partial \omega$, i.e.  $\uu \cdot n  = -\partial p /\partial n$. 
Now consider the stochastic interpretation of the Poisson equation 
 $$
 -\Delta p =  \mu .
 $$
Considering  that $\mu$ is a probability measure, we consider a random variable $X\in \omega$  which follows  the law of  probability  with density $\mu$.
We now consider a  Brownian motion stemming from $X$. Its path crosses $\partial \omega$ for the first time at $Y$. The random variable $Y\in \partial \omega$ is known to follow the probability law with density $-\partial p / \partial n$ on $\partial\omega$ (see~\cite{doob}).

The idea is therefore to redistribute the exceeding mass of the saturated zone in the following way (see Fig.~\ref{fig:proj_stoch}): For each saturated cell, a random walk is started, which transports the exceedind mass $m = (\rho - 1)|C|$, where $|C|$ is the measure of the cell. When this random walk encounters a non-saturated cell, it gets rid of as much mass as it can, and continues as long as the transported mass is not fully distributed. When all the saturated cells have been treated, the obtained density $\rho^{n+1}$ is admissible.

\begin{figure}[h!]
\begin{center}
\includegraphics[width=0.6\linewidth]{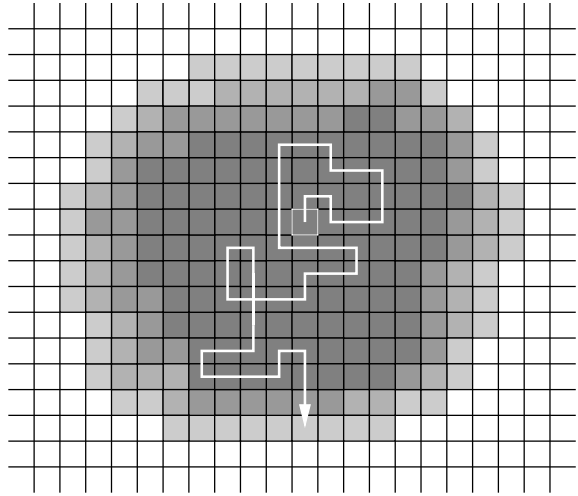}\\
\caption{Stochastic projection of the transported density}
\label{fig:proj_stoch}
\end{center}
\end{figure}


Let us notice that the gradient flow structure when $\UU=-\nabla D$ would also allow for another time-discretized algorithm, namely the proximal (or minimizing movement, or Jordan-Kinderlehrer-Otto) one instead of the prediction-correction one. This would lead to a sequence of minimization problems involving the computation of the distance $W_2$. The formulation by means of the transport plans, that we did not develop here for the sake of readability, transforms these problems into linear programming problems, that could a priori be solved through a simplex algorithm. Yet, it  turns out to be so  slow  that 2D problems could not be efficiently solved in this way.

\bigskip

\para{Numerical analysis}
\label{sec:num_an_macro}
The numerical analysis of the transport part of the algorithm is standard, but the projection part is more delicate, and we are not able to provide a rigorous  error estimate for the stochastic scheme we propose.  
 Let us simply say here that, beyond the unformal justification of the approach given previously,  the asymptotic behavior of the so-called Diffusion Limited Aggregation (DLA) process (see~\cite{LevPer}), and the link between the distance based on the energy norm between measures (see~\cite{GusSak} or~\cite{doob}) and the Wasserstein distance  give some arguments to support  the chosen strategy.

\section{From micro to macro}
\label{sec:micro_macro}
Both microscopic and macroscopic model express the same assumptions at their respective levels: a spontaneous velocity is given, and the system evolves according to a velocity which is the closest to the spontaneous one among all feasible velocities, in a least square sense. Yet, the macroscopic model was not built from the microscopic one by a rigorous homogenization process. We describe here some obstacles to such a process, to shed light on the deep differences between both settings, in spite of formal similarities. 

\bigskip
\para{Maximal density}
\\
First of all, let us point out that the notion of maximal density is somewhat ambiguous as far as the microscopic model is concerned. Let us consider the case of identical radii ({\rm monodisperse} situation). The maximal packing density for identical disks is known to be $\rho_{max} = \pi / 2\sqrt{3} \approx 0.9069\dots$, and corresponds to the triangular lattice (see Fig.~\ref{fig:triang}).
\begin{figure}[h!]
\begin{center}
\includegraphics[width=0.6\linewidth]{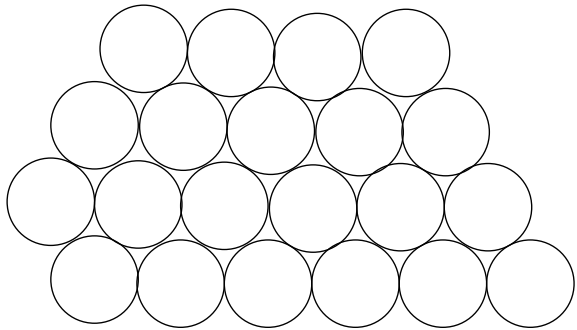}\\
\caption{Triangular lattice}
\label{fig:triang}
\end{center}
\end{figure}
Yet the actual density of moving collections of rigid disks is generally strictly less than this maximal value, which is only attained for this very particular configuration. As an example, in the evacuation situation represented in Fig.~\ref{fig:densites_micro}, the mean density upstream the exit door ranges typically between $0.85$ and $0.87$.
\begin{figure}[h!]
\begin{center}
 \begin{minipage}[c]{0.45\linewidth}
 \includegraphics[height=\linewidth]{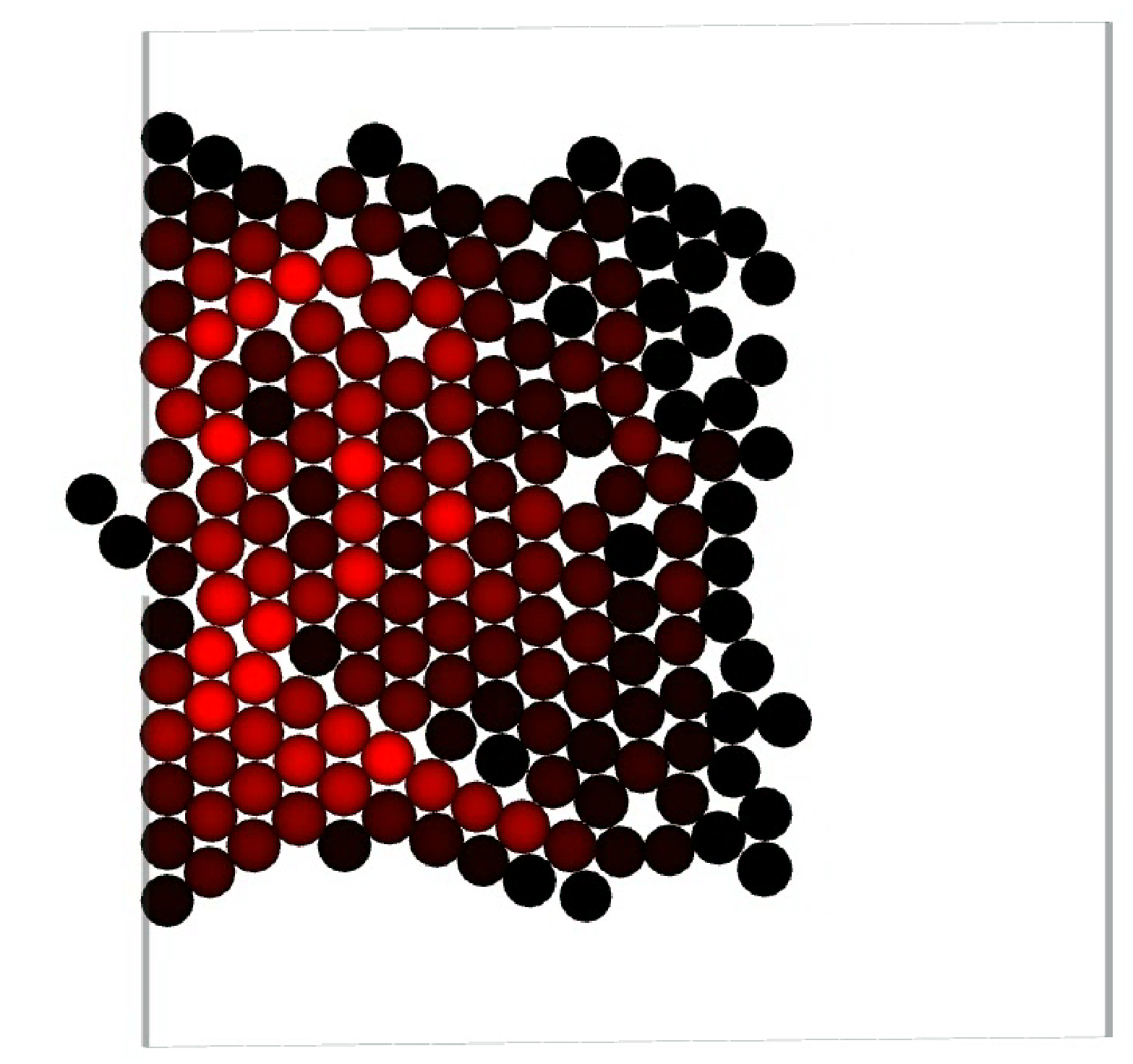}
 \end{minipage}
\hspace*{0.7cm} \begin{minipage}[c]{0.45\linewidth}
 \includegraphics[height=\linewidth]{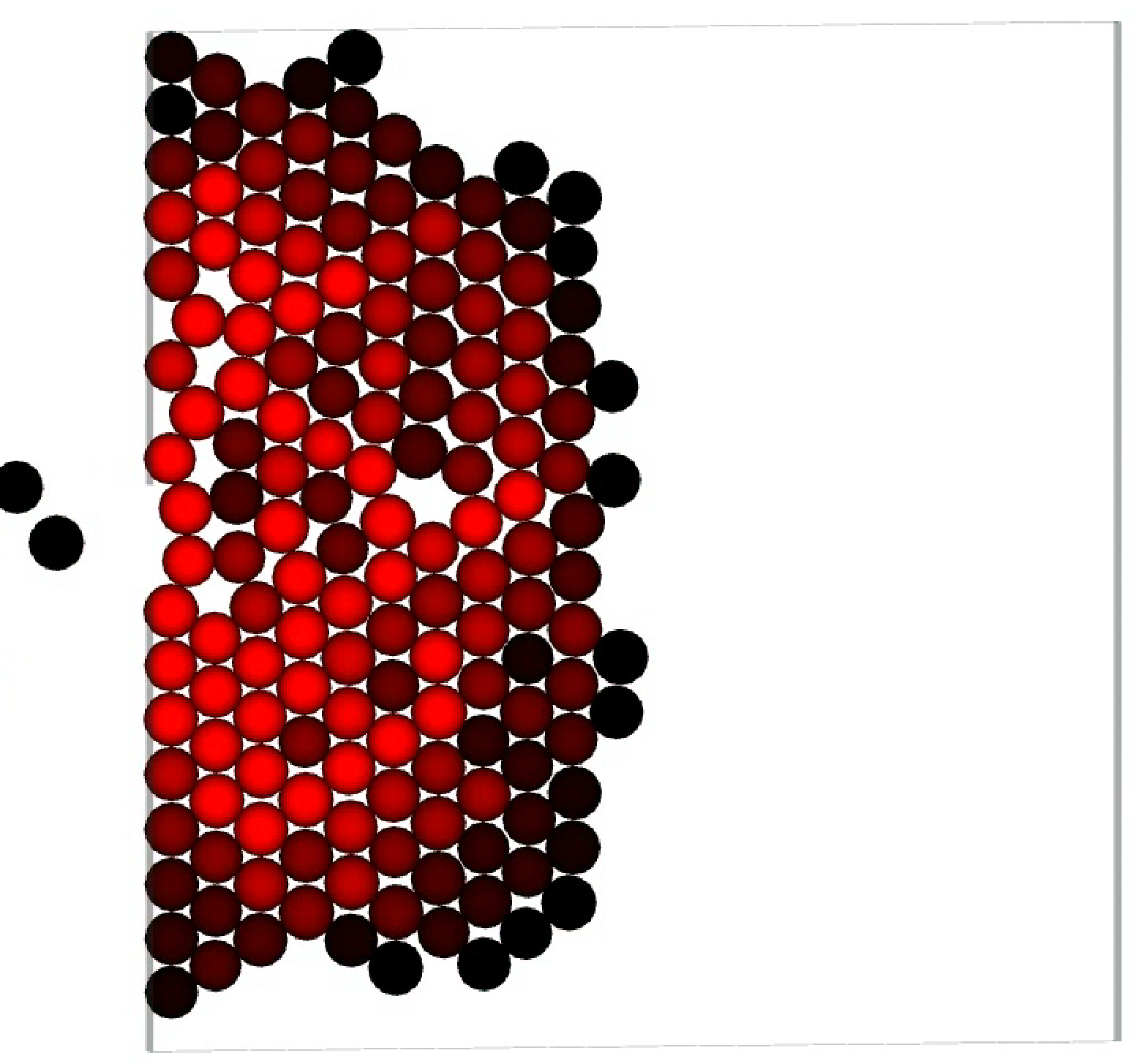}
 \end{minipage}
\caption{Maximal densities in the microscopic setting.}
\label{fig:densites_micro}
\end{center}
\end{figure}
The fact that the actual density is strictly less than the maximal one does not mean that the flow is unconstrained (as the macroscopic setting would suggest). Those considerations call for a clear identification of configurations which saturate the constraint. Such a notion is proposed for the three-dimensional situation  in~\cite{torquato}, in a Nash equilibrium spirit:

\begin{center}
{\em 
We say that a particle (or a set of contacting particles) is jammed if it cannot be translated while fixing the positions of all of the other particles in the system.}
\end{center}
Note that this property  is defined  as {\em local jamming} in~\cite{torquato}.
It is tempting to consider as maximal in some sense any density corresponding to such configurations, for which there are no free disks, so that constraints are activated everywhere. The triangular lattice is clearly jammed, but so is the cartesian lattice ($\rho = \pi / 4  \approx 0.79$), and it is possible to build looser jammed configuration (see Fig.~\ref{fig:triangloose}, with $\rho = \pi\sqrt{3} / 8 \approx 0.68$).
\begin{figure}[h!]
\begin{center}
\includegraphics[width=0.6\linewidth]{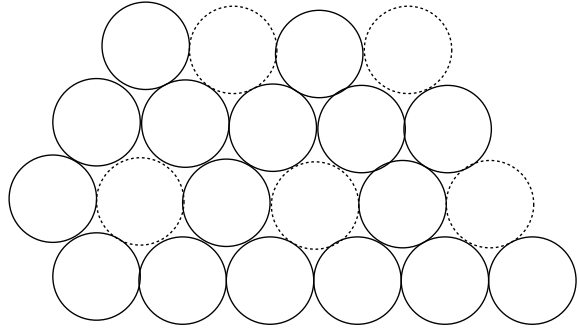}\\
\caption{Loose triangular lattice}
\label{fig:triangloose}
\end{center}
\end{figure}

\para{Constraints on the velocity}
Beyond this fuzzy definition of maximal packing, the set of feasible velocities in the microscopic setting is strongly dependent on the local structure and not only on the density, a feature which is lost in the macroscopic approach. 
Let us first consider the triangular lattice represented in Fig.~\ref{fig:triang}. Having the number of disks going to infinity and the radius going to zero, the density  (characteristic function of the solid phase) converges to the uniform density $\rho_{max}$.
Microscopic feasible velocities can be defined  in the solid phase, and one may wonder what are the corresponding feasible  velocities in the macroscopic limit. Such velocities are clearly constrained in  3 directions. Using self-evident notations $\ee_0$, $\ee_{\pi/3}$ and  $\ee_{2\pi/3}$ to design the principal directions of the lattice, one obtains at the limit some sort of unidirectional expansion  constraints in those directions, which can be written
$$
\ee\cdot \nabla \uu \cdot \ee\geq  0 \virg\hbox{  with } \ee = \ee_0\virg \ee_{\pi/3}\hbox{ or }\ee_{2\pi/3}.
$$
It still allows for non-rigid motions at the limit: Fig.~\ref{fig:triangmove} represents a feasible velocity field at the microscopic level, whose macroscopic counterpart is 
$$
\uu = \left ( \begin{array}{c} \sqrt{3}x \\  - y \end{array} \right ).
$$
Note that, although the microscopic velocity field is tangent to the boundary of $K$, the corresponding  macroscopic field is strictly expansive ($\nabla \cdot \uu > 0$).

\begin{figure}[h!]
\begin{center}
\includegraphics[width=0.6\linewidth]{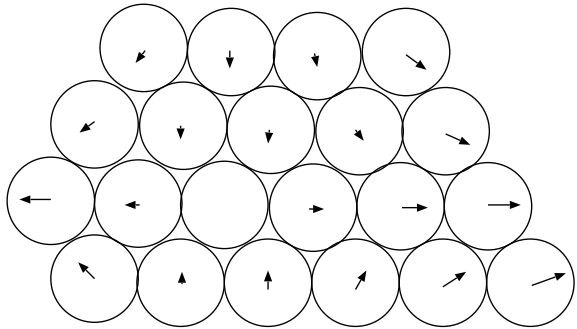}\\
\caption{Example of feasible velocity}
\label{fig:triangmove}
\end{center}
\end{figure}

\medskip

\para{Jams}
\\
The considerations which we presented above have a crucial consequence on the behavior of both approaches: the microscopic model has the ability to reproduce blocked jams (which are observed in practice), whereas the macroscopic one does not. 
The latter property is a consequence of the maximum principle. Let us
consider  the situation represented in Fig.~\ref{fig:jammac}, with a
saturated zone ($\rho\equiv 1$) upstream the exit door. The desired
velocity $\UU$ is assumed to point to the exit door, so that
$\nabla\cdot  \UU < 0$. The actual velocity is $\UU - \nabla p$, where $p$ solves a Poisson problem in the saturated zone, with homogeneous Neumann condition on the walls, and homogeneous  Dirichlet boundary conditions at the exit and on the interface with the non-saturated zone. By virtue of the maximum principle, $p$ is nonnegative over the saturated zone, so that the velocity correction $-\partial p/\partial n  $ is non-negative on the door: people exit {\em quicker} than they would if there were no congestion (they are always pushed out by other people behind them). As a consequence, if the desired velocity field tends to move people out of the room, the evacuation process will never stop before the room is empty.

\begin{figure}[h]
\psfrag{p0}[l]{$p=0$}
\psfrag{mDeltap}[l]{$-\Delta p=-\nabla \cdot \UU \geq 0 $}
\psfrag{dpdn}[l]{$\partial p /\partial n=0$}
\psfrag{Sat}[l]{Saturated zone}
\begin{center}
\resizebox{0.95\textwidth}{!}  {\includegraphics{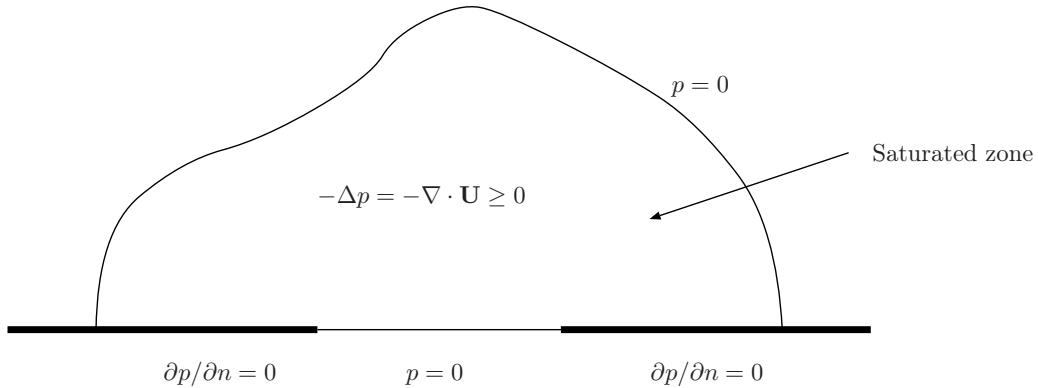}}
\caption{Macroscopic evacuation}
\label{fig:jammac}
\end{center}
\end{figure}

The microscopic situation is quite different. The pressures are still nonnegative by definition of the saddle-point formulation~(\ref{eq:spmic}), but they are likely to act in the ``wrong''  direction for the individuals which are the closest to the door, as soon as they form an arch (see Figure~\ref{fig:jammic}). 

\begin{figure}[h]
\begin{center}
\resizebox{0.8\textwidth}{!}  {\includegraphics{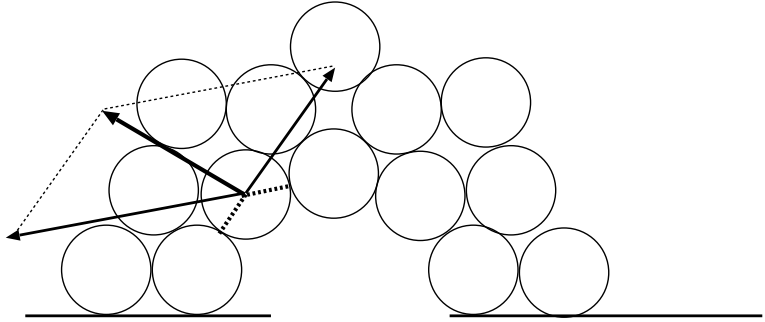}}
\caption{Microscopic evacuation (arches)}
\label{fig:jammic}
\end{center}
\end{figure}

\bigskip

\para{Mathematical issues (convergence or non-convergence)}
\\
The precise way for setting convergence questions when passing from the microscopic to the macroscopic model needs to associate a density to every microscopic configuration. An easy way to do that is to associate to every non-overlapping disks configuration the uniform measure (with unit density) on the union of those disks. In this way we obtain a density $\rho$ obviously satisfying $\rho\leq 1$ but we also know that, when we let the size of the disks go to zero (with the number of disks increasing consequently to infinity) we cannot obtain {\it any} density $\rho\leq 1$ in the macroscopic limit. In particular in dimension 2 it is impossible to go beyond $0.91$, since we know that the maximal density is realized by the triangular lattice. 

This suggests that it is better to rescale the approximating $\rho$ by the maximal density, which is $1$ in dimension one, $\pi/2\sqrt{3}$ in dimension two\dots 

Now the question is the following: take a sequence of initial data $\rho^0_N$, i.e. the densities associated (in the way we described above) to a sequence of microscopic configurations with $N$ non overlapping disks of radius $r$. Let $N\to\infty$ and $N\approx r^{-d}$ and consider the limit density $\rho^0$. Can we say that, for fixed velocity field $U$, the constrained evolution stemming in the microscopic model from $\rho^0_N$ converges to that obtained in the macroscopic one from $\rho^0$?

The answer can be expected to  be
 positive if $d=1$, but the situation is much more complicated for $d\geq 2$. Actually, the unidimensional case gives no ambiguity between jammed and maximal density configuration, which is not the case in higher dimension.

For instance one can consider a sequence of jammed cartesian lattices in a two-dimensional domain $\OO$ (a square, for instance), with a concentrating vector field $U$ (for instance $\UU(x)=-x$). In this situation, the configuration is completely blocked, and the evolution gives, for any time $t>0$, the constant density $\rho^0_N$. On the other hand, the limit as $N\to\infty$ is a constant density, but it does not activate the constraints. This is due to the fact that the density realized by the cartesian grid is strictly smaller than that of the triangular lattice, which was taken as a reference. Hence, in the evolution, $\rho^t$ would differ from $\rho^0$ and move towards a more concentrated configuration.

But this (jammed configuration which are not of maximal density) is not the only source of problems in the micro-macro limit. Actually, one can consider a similar example with a sequence of triangular grids. In such a case, both $\rho^0_N$ and the limit $\rho^0$ activate and saturate the constraints. This induces some constraints on the admissible velocities but, as we saw in the paragraph devoted to these constraints, they act differently in the microscopic and in the macroscopic models. In particular, the limits of the admissible velocities for the microscopic problem are those vector fields which are ``unilaterally incompressible'' in the three main direction of the lattice, 
which is not at all the same as imposing only a positive divergence. The actual limit constrained space at the macroscopic level necessitates obviously to account for the microscopic structure of the contact network. Beyond packing fraction,  different types of  {\em order metrics} have been introduced in~\cite{donev}  to give a better description of the local structure (see also~\cite{torquato} for a recent review on this matter), but the use of those concepts for evolution problem is still widely open.

This shows some deep differences, for $d\geq 2$, between the limit of the microscopic case and the macroscopic one. In the case of a gradient vector field $\UU$ we could express this issue, thanks to the gradient-flow formulation, in terms of the associated dissatisfaction functionals. Some general results give conditions to translate a variation form of convergence on these functionals (called $\Gamma-$convergence) into a convergence of the associated flows, but we do not want to insight this question in details because of its complexity; it would go out of the scope of this paper and it is moreover matter of an ongoing work.
 
%
%
%

\section{Possible extensions}
\label{sec:ext}

\subsection{Strategies}

We specify here how the model we presented can be improved in order to account for more elaborate individual behaviors than the purely reptilian one on which we built our approach.\\

\para{Microscopic setting}
We aim here at integrating the fact that individuals are likely to   elaborate complex strategies to escape a building. For example, in
case of congestion, they may decelerate or try to avoid the jam, instead of
keeping pushing inefficiently. The velocity
 of a person becomes then dependent upon the position of people he can see in front
 of him.
 Such a strategy can be defined as follows:
We define the set $\textmd{N}_i$ (see Fig. \ref{fig:Ni}) containing persons
 who are near and visible to the 
 individual $i$: $$\textmd{N}_i=\left\{\  j,\ 
 |\qq_i-\qq_j|<2 r + \ell_{prox} , \  \dd_i \cdot \ee_{ij} \geq
 \cos \alpha  \right\} , $$ 
where $ \dd_i =   \UU_0(\qq_i)/|\UU_0(\qq_i)| $
and $ \ee_{ij} = (\qq_j-\qq_i) /|\qq_j-\qq_i|$. The constant  $\alpha$
is taken equal to the half angle of view ( i.e. $ \alpha  \simeq 60 ^\circ $). 
Two choices are possible for the individual $i$ if his neighbours
 (belonging to $\textmd{N}_i$) walk slower than him.

First,  he can
 decelerate instead of going through the crowd. In this case, his
 speed $s^n_i$ (i.e. his 
 velocity's norm)  at time $t^{n}$ is made dependent upon his neighbours'
 behavior at time $t^{n-1}$. More precisely, it is computed as a barycenter of his
 neighbours' speeds, weighted by their relative positions.

 Otherwise, he can also be in a hurry and prefer changing
 his way instead of slowing down. In that case, 
if there exists another clear way through the set $\textmd{N}_i $,
 he follows it while keeping his desired speed, or else he will go round the
 group 
 $\textmd{N}_i$. Among all directions allowing him to do so, he
 chooses $\dd_i^{new} $ the 
 closest to the one he wanted (see
 Fig. \ref{fig:bypass}).
 
\begin{figure}
\begin{center}
\subfigure[Illustration of the set $\textmd{N}_i$.]{
\psfrag{v}[c]{ $\dd_i$}
\psfrag{a}[c]{$\alpha$}
\psfrag{d}[c]{$\ell_{prox}$}
\includegraphics[width=0.45\textwidth]{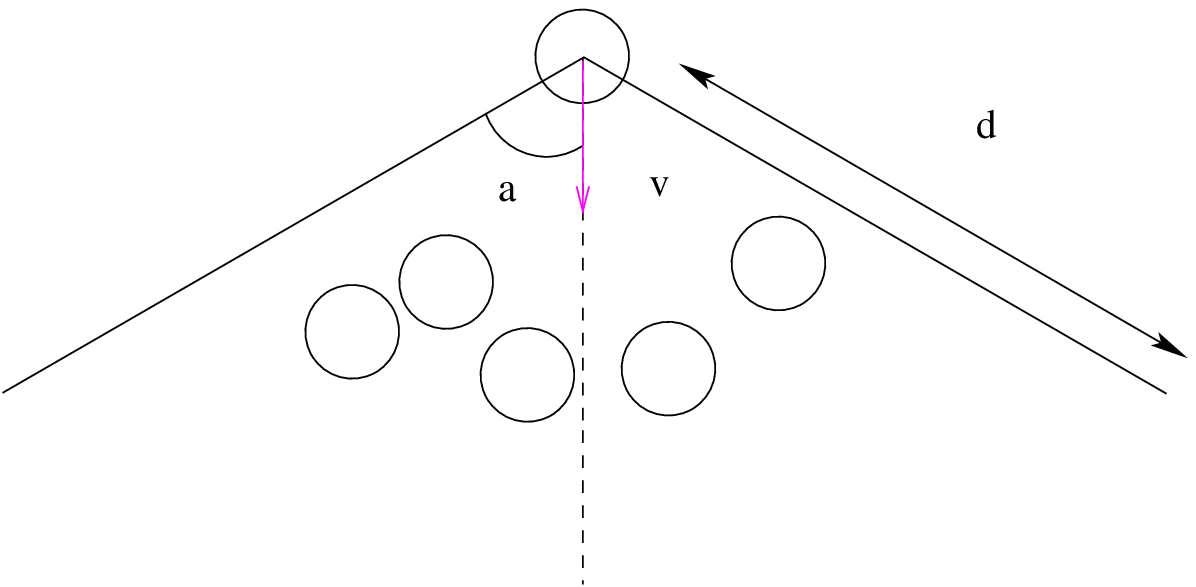}
\label{fig:Ni}}
\subfigure[Bypass strategy.]{
  \psfrag{o}[c]{\small $\dd_i$}
  \psfrag{n}[c]{\small$\dd_i^{new}$}
  \psfrag{ql}[c]{$\qq_{j^l}$}
  \psfrag{qr}[c]{$\qq_{j^r}$}
 \psfrag{e}[c]{$\ee_{ij^l}$}
 \psfrag{ep}[c]{$\ee_{ij^l}^\perp$}
 \psfrag{f}[c]{$\ee_{ij^r}$}
 \psfrag{fp}[c]{$-\ee_{ij^r}^\perp$}
  \includegraphics[width=0.45\textwidth]{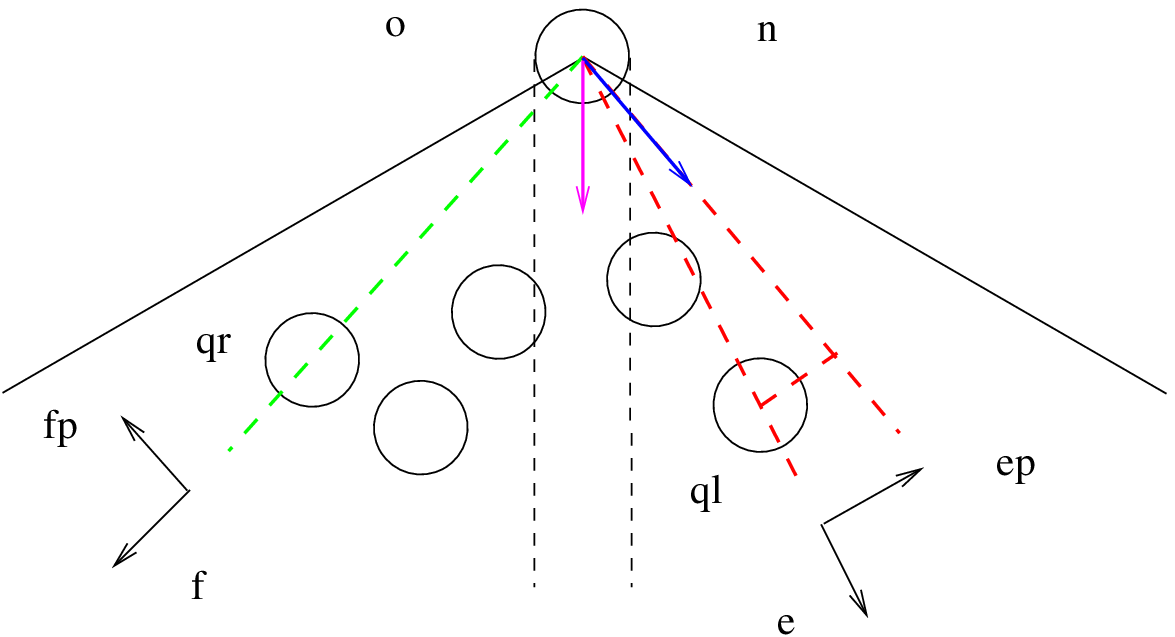}
\label{fig:bypass}}
\caption{Integrating strategies.}
\end{center}
\end{figure}

Numerical simulations with these strategies are presented  in~\cite{Pedjv}. If every pedestrian prefers decelerating, there is typically no
contact between people and rarefaction
waves are observed. On the contrary, the avoiding strategy  leads to situation where the constraints are activated.


   \bigskip
   
\para{Macroscopic social forces}
Integrating  strategies in the macroscopic setting is  more delicate, as  it is less natural to follow  the motion of a single individual in the crowd. 
%
 However, it is possible to model the influence of  local density on the behavior of the crowd by changing the desired velocity. More precisely, it is natural to assume that when people arrive upstream a crowded area, they tend to decrease their desired velocity. 
 In the spirit of~\cite{Hug1,markow}, 
a two dimensional generalization of road traffic models can be integrated in the model: at time $t$, people located at $x \in \OO$ have a desired velocity $\UU_\rho(x)$ with the same direction as the initial desired velocity $\UU(x)$ and a norm that decreases when $\rho$ increases:
\begin{equation}
\UU_\rho(x) = \alpha(\rho) \UU(x),
\end{equation}
where $\alpha(\rho)$ goes monotonically from $1$ to $0$ as $\rho$ goes from $0$ to $1$ (saturation value). 
The evolution of the crowd then obeys the same principles as before, namely:
\begin{equation}
\label{eq:strat_macro}
\left\{ \begin{array}{rcl}
\partial_t \rho + \nabla \cdot (\rho \uu) & = & 0\\
\uu & = & P_{C_\rho} \UU_\rho.
\end{array} \right.
\end{equation}
It can be proved that the density transported by $\UU_\rho$ 
stays in $K$: the projection step in the evolution equation becomes useless. Yet, if $\alpha(1) >0$ (people continue to push even at saturation density), the congestion constraint is likely to be activated, and the approach  we propose in this paper can be adapted to this situation. Note that, from a modeling   standpoint, it is natural to use {\em downstream} information to determine the desired velocity (people adapt their velocity according to what they see, i.e.\ to the density in front of them, which is downstream their desired direction), whereas, once the velocity is determined,  the transport can be  performed using an standard upwind  scheme.


\subsection{Multi-component populations}

We investigate here the possibility to account for different types of individuals, who might have different, and possibly antagonist, strategies.  In the microscopic setting, because of the Lagrangian character of the approach, this can be done without any change in the mathematical structure: it is sufficient to alleviate the assumption that the desired velocity depends on the position only, but can be defined differently depending of individuals. 
In the macroscopic setting, this is less straightforward, but it can  be formulated in the following way: consider two populations $1$ and $2$, with associated densities $\rho_1$ and $\rho_2$ and desired velocity fields $\UU_1$ and $\UU_2$, we consider that in the saturated zone $[\rho_1 + \rho_2 = 1]$, desired velocities are perturbed  by  a common correction velocity $\ww$ which ensures that the constraint $\rho_1 + \rho_2 \leq  1$ remains satisfied, and we define this velocity as the one that minimizes the $L^2$ norm.
The model can be  written
\begin{equation}
\left\{ \begin{array}{rcl}
\partial_t \rho_1 + \nabla \cdot (\rho_1 ( \uu_1 + \ww)) & = & 0\vseq\\
\partial_t \rho_2 + \nabla \cdot (\rho_2 ( \uu_2 + \ww)) & = & 0\vseq\\
\ww + \nabla p &=& 0 \vseq\\
\nabla \cdot \left ( \rho _1 \UU_1 + \rho_2 \UU_2 + \ww\right ) & \geq  & 0.
\end{array} \right.
\end{equation}
where $p$ vanishes in the non-saturated zone $[\rho_1 + \rho_2 < 1]$, and verifies the complementarity condition
$$
\int (\rho_1 \UU_1 + \rho_2 \UU_2) \cdot \nabla p = 0.
$$
Note that, in case $\UU_1$ and $\UU_2$ are both gradients of   dissatisfaction functions (one for each population), the system  presents a gradient-flow structure in the Wasserstein setting (see~\cite{DMMR} for a similar problem in the context of cell migration).

\subsection{Nash equilibrium approach}
We give here some hints on another possibility to define the  actual velocity, in a Nash equilibrium spirit. Note that this approach is likely to change the mathematical nature of the evolution problem. In particular, the instantaneous  velocity is no longer defined in a unique way, but as any (among infinitely many) equilibrium in the following sense:
%
The actual generalized velocity $\uu =(\uu_1,\uu_2,\dots,\uu_N)$ is such that, for any $i$, $\uu$ considered as a function of $\uu_i$ only minimizes the distance to $\UU$  among feasible velocities, i.e.
$$
\uu \in \argmin _{\vv \in C_\qq^i(\uu)} \frac 1 2 \abs {\vv- \UU_i } ^2
$$
where $C_\qq^i(\uu)$ is the set
$$
C_\qq^i(\uu) = \set { \vv \in \setR^2\virg (\uu_1,\dots,\uu_{i-1},\vv,\uu_{i+1},\dots,\uu_N) \in C_\qq}.
$$
As soon as two disks are in contact and push against each other, there are infinitely many solutions to this problem, as illustrated by the situation of Fig.~\ref{fig:mashmic}: the desired velocity of person $1$ is $1$ (in the horizontal direction), whereas person $2$ tends to stay still. 
In the previous approach ($\ell^2$ projection onto the cone of feasible velocity), the actual velocity is $1/2$ for both. In the present approach, any diagonal couple $(\alpha, \alpha)$, with $\alpha \in [0,1]$, realizes  an instantaneous   Nash  equilibrium.

\begin{figure}[h!]
  \psfrag{1}[c]{$1$}
  \psfrag{2}[c]{$2$}
  
\begin{center}
 \includegraphics[width=0.6\linewidth]{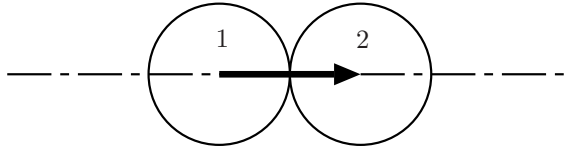}
\caption{Two disks, Nash approach.}
\label{fig:mashmic}
\end{center}
\end{figure}

\bigskip

In the macroscopic framework, the problem  can be described as follows:
For any $\omega \subset \setR^2$, one prescribes that 
$$
\uu_{|\omega} \in \argmin_{\vv \in C_\rho^\omega(\uu)} \frac 1 2  \left \| \vv_{|\omega} - \UU_{|\omega} \right \|_{L^2(\omega)}
$$
with
$$
C_\rho^\omega(\uu) = \set{ \vv \in L^2(\Omega)\virg \vv \oplus \uu_{|\setR^2 \setminus\overline\omega} \in C_\rho
}
$$
As indicated previously, this approach changes the mathematical nature of the evolution problem, which calls for further developments. Let us simply say that the solutions we built in the previous sections are particular solutions to this new problem, among infinitely many others.

\section{Numerical experiments}

\subsection*{Micro-macro comparison}

The first problem of the comparison between these models lies in the initial configuration: given an initial configuration for the disks in the microscopic setting, it is difficult to choose the macroscopic density that best fits this configuration. As long as uniform densities are considered, we adopt the following method: we first estimate the mean density in the initial microscopic configuration, and then normalize it with the maximal density of disks we encounter throughout the microscopic evolution. In Fig.~\ref{fig:init_max_density}, we present an example of such a computation.

\begin{figure}[h!]
\begin{center}
 \begin{minipage}[c]{0.4\linewidth}
 \includegraphics[width=\linewidth]{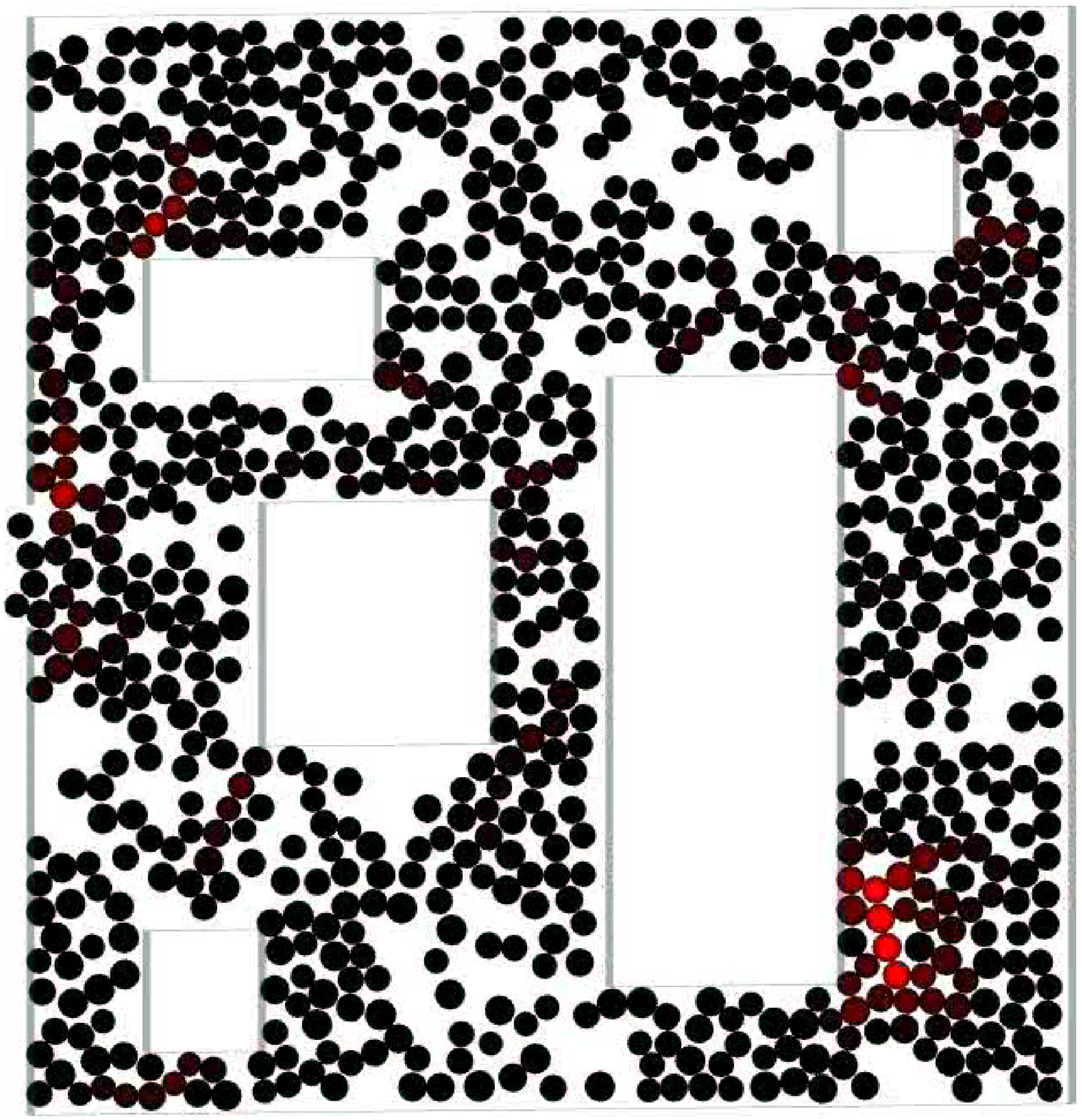}
 \end{minipage}
\hspace*{0.7cm} \begin{minipage}[c]{0.3\linewidth}
 \includegraphics[width=\linewidth]{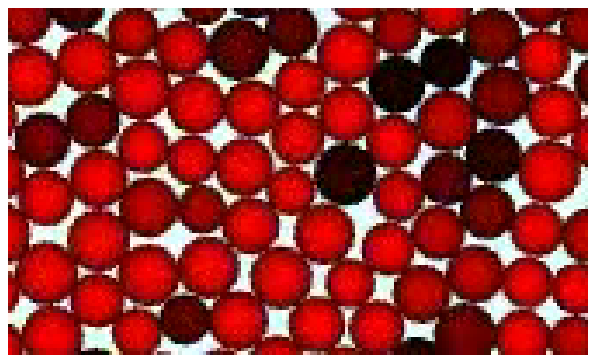}
 \end{minipage}
\caption{Calculation of the macroscopic density: the initial  density $0.53$ (on the left) is normalized by the maximal density encountered $0.81$ (on the right) to obtain a macroscopic density of $0.65$.}
\label{fig:init_max_density}
\end{center}
\end{figure}

We then let evolve both systems and compare the configurations at equivalent time steps. The formation of the saturated zones is quite the same in microscopic and macroscopic settings (see Fig.~\ref{fig:5obs_01}). However, as pointed out in Section~\ref{sec:micro_macro}, the behavior of the two models at the exit is very different, in particular the evacuation is faster in the macroscopic model (see Fig.~\ref{fig:5obs_15}).

\begin{figure}[h!]
\begin{center}
 \begin{minipage}[c]{0.45\linewidth}
 \includegraphics[height=\linewidth]{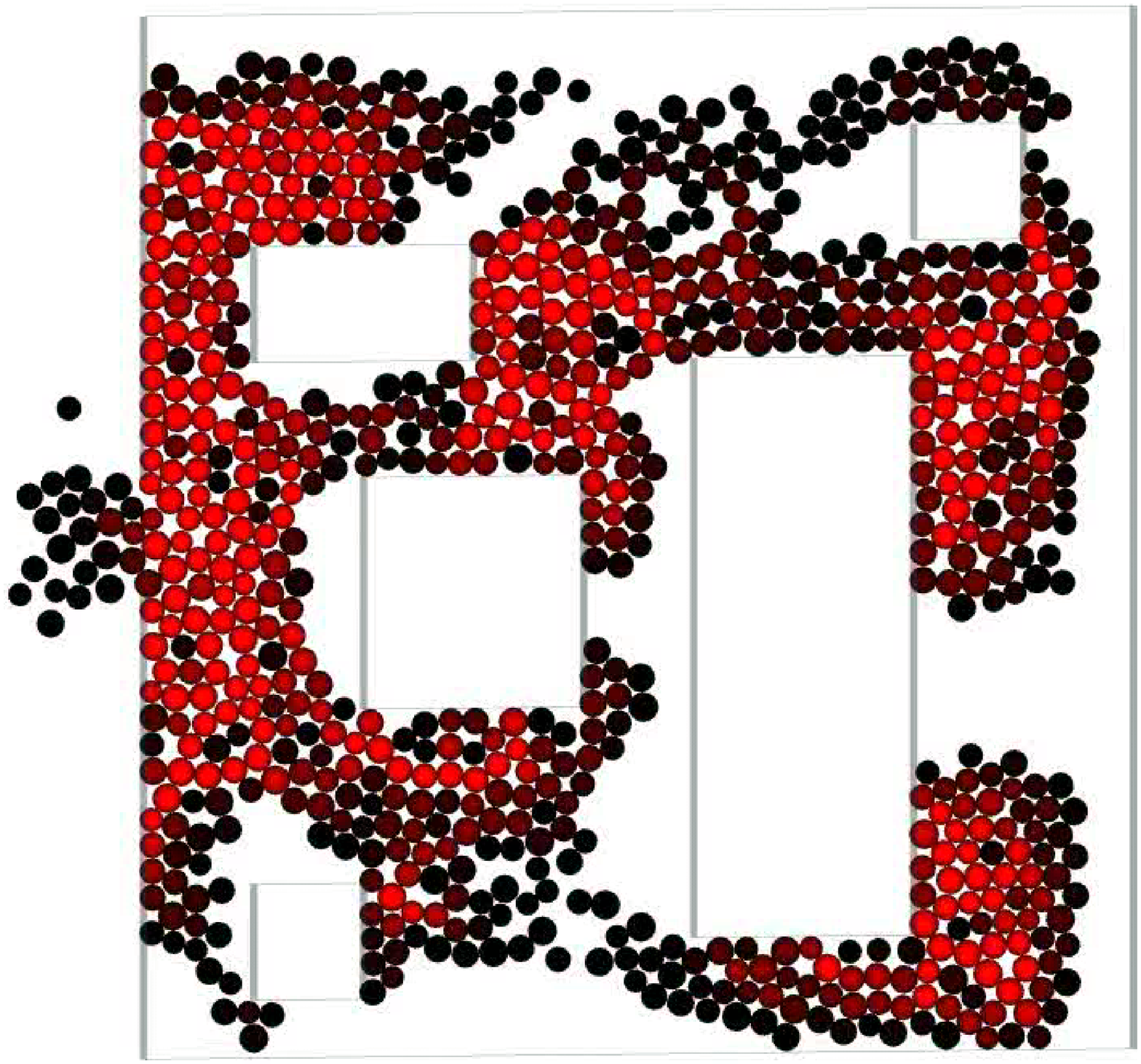}
 \end{minipage}
\hspace*{0.7cm} \begin{minipage}[c]{0.46\linewidth}
 \includegraphics[height=\linewidth]{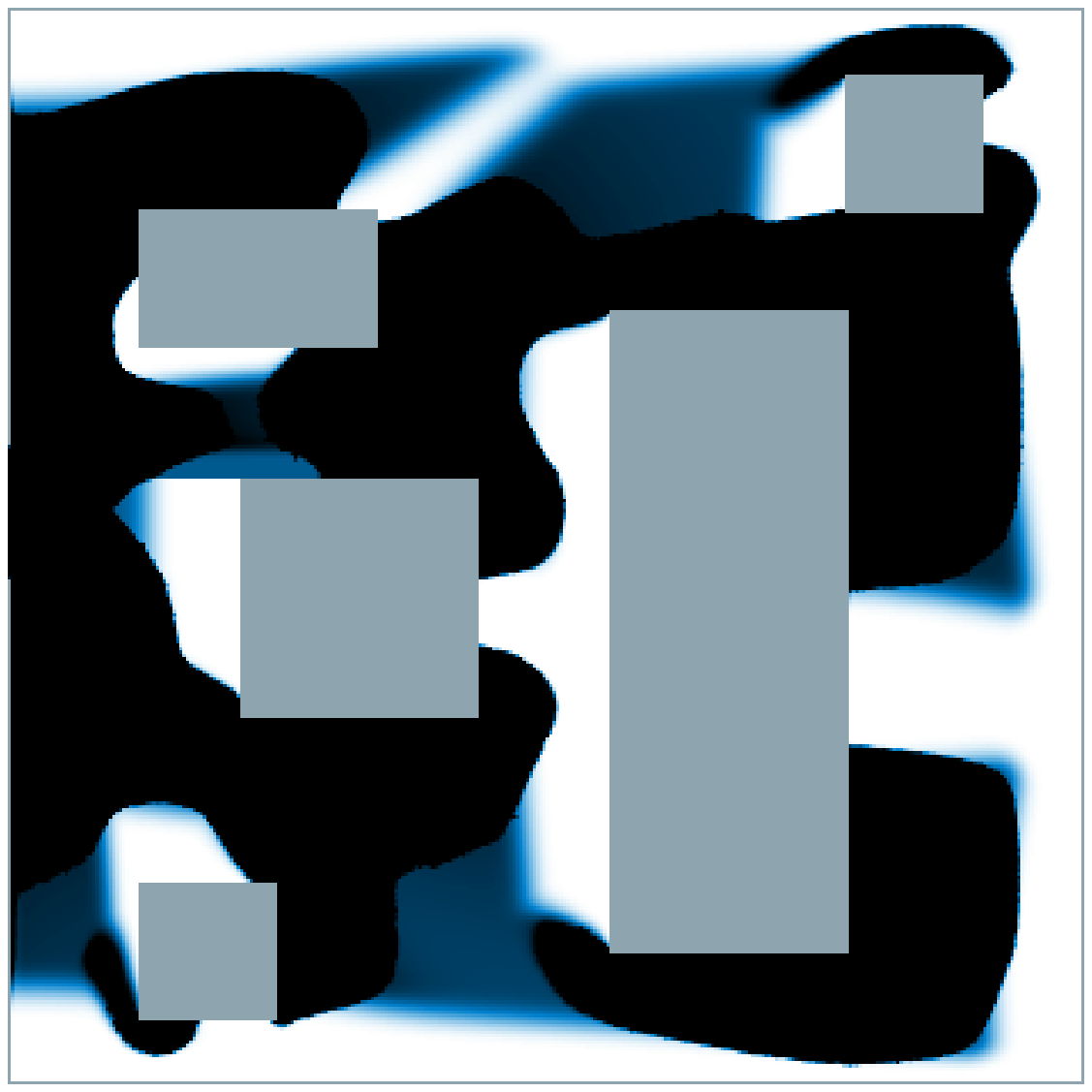}
 \end{minipage}
\caption{Formation of the saturated zones in the microscopic model (left) and the macroscopic one (right).}
\label{fig:5obs_01}
\end{center}
\end{figure}

\begin{figure}[h!]
\begin{center}
 \begin{minipage}[c]{0.45\linewidth}
 \includegraphics[height=\linewidth]{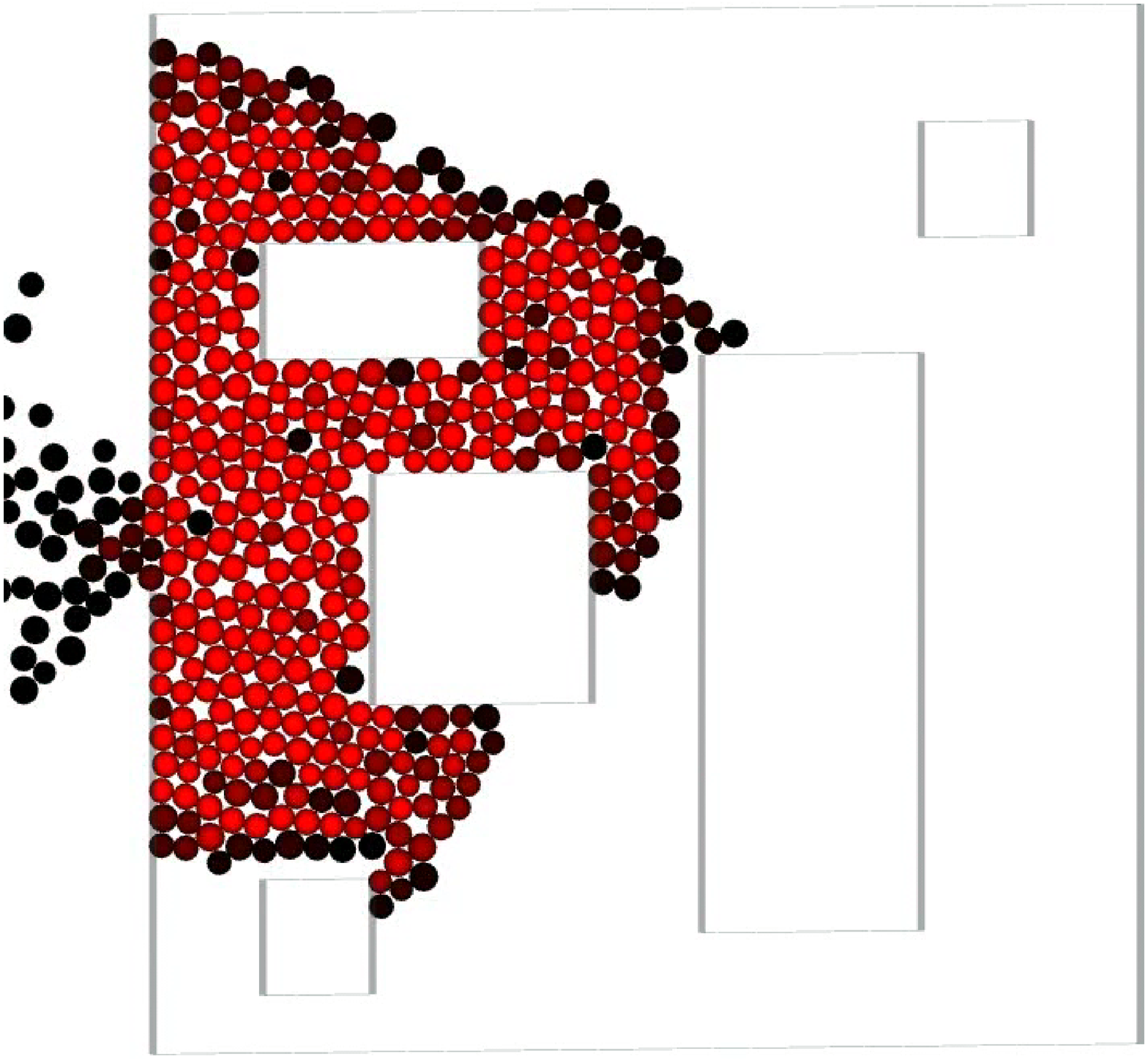}
 \end{minipage}
\hspace*{0.7cm} \begin{minipage}[c]{0.46\linewidth}
 \includegraphics[height=\linewidth]{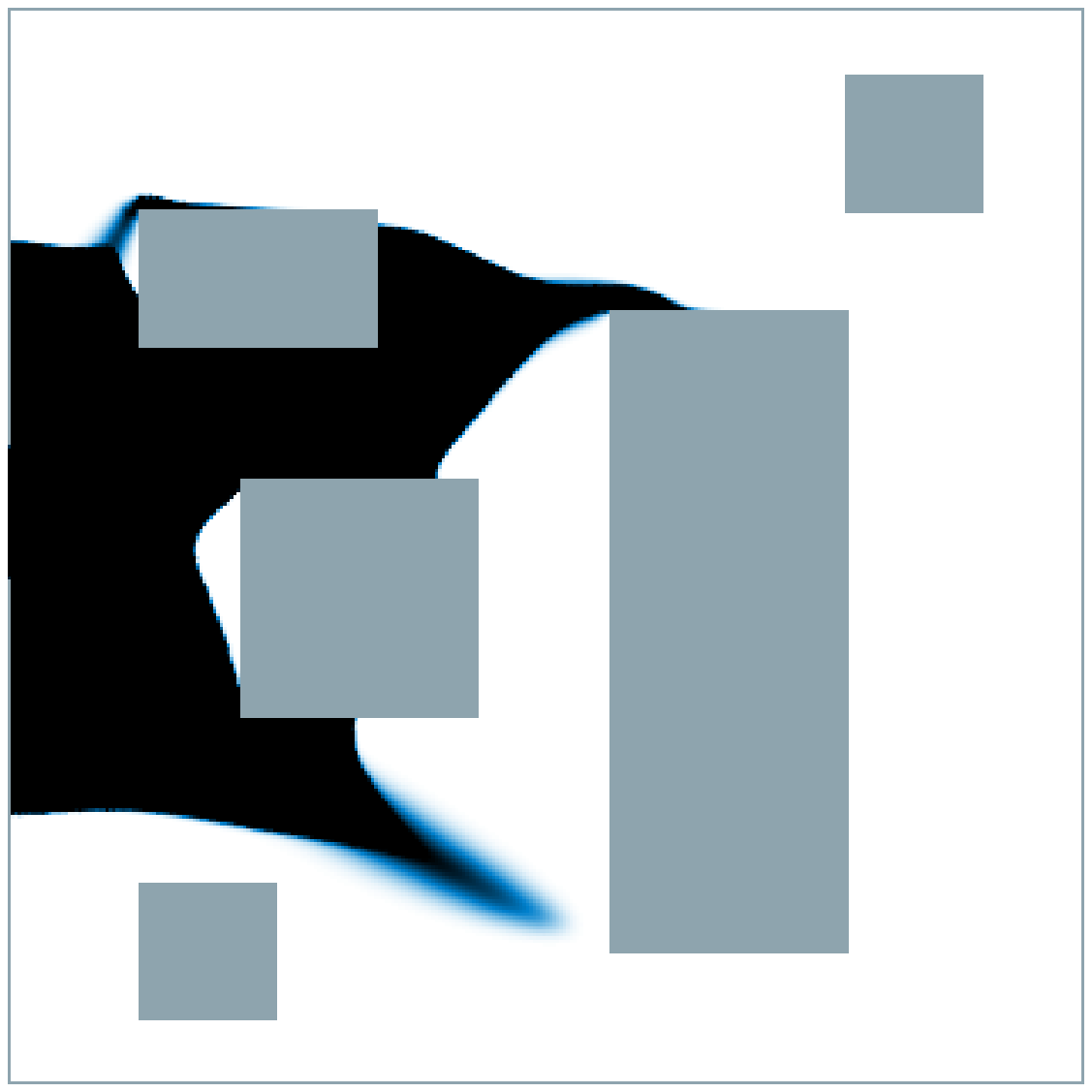}
 \end{minipage}
\caption{Comparison of the evacuation for both models.}
\label{fig:5obs_15}
\end{center}
\end{figure}

\subsection*{Pressure field}

In the microscopic model, the pressure exerted between people appears naturally as a Lagrange multiplier, and is computed using Usawa algorithm. The macroscopic pressure field does not appear explicitly in the numerical scheme, but it can be recovered by estimating during the projection step a discrete equivalent of the odometer function (see Section~\ref{sec:num_an_macro}): on each cell, we define the pressure as the total mass emitted by this cell during the stochastic projection. In Fig.~\ref{fig:pressions}, we represented the microscopic pressure and the isovalues of the macroscopic pressure field upstream an exit door. In both settings, the pressure field is maximal in the middle of the saturated zone, and decreases at the entrance and the exit of this zone.

\begin{figure}[h!]
\begin{center}
 \includegraphics[height=0.85\linewidth]{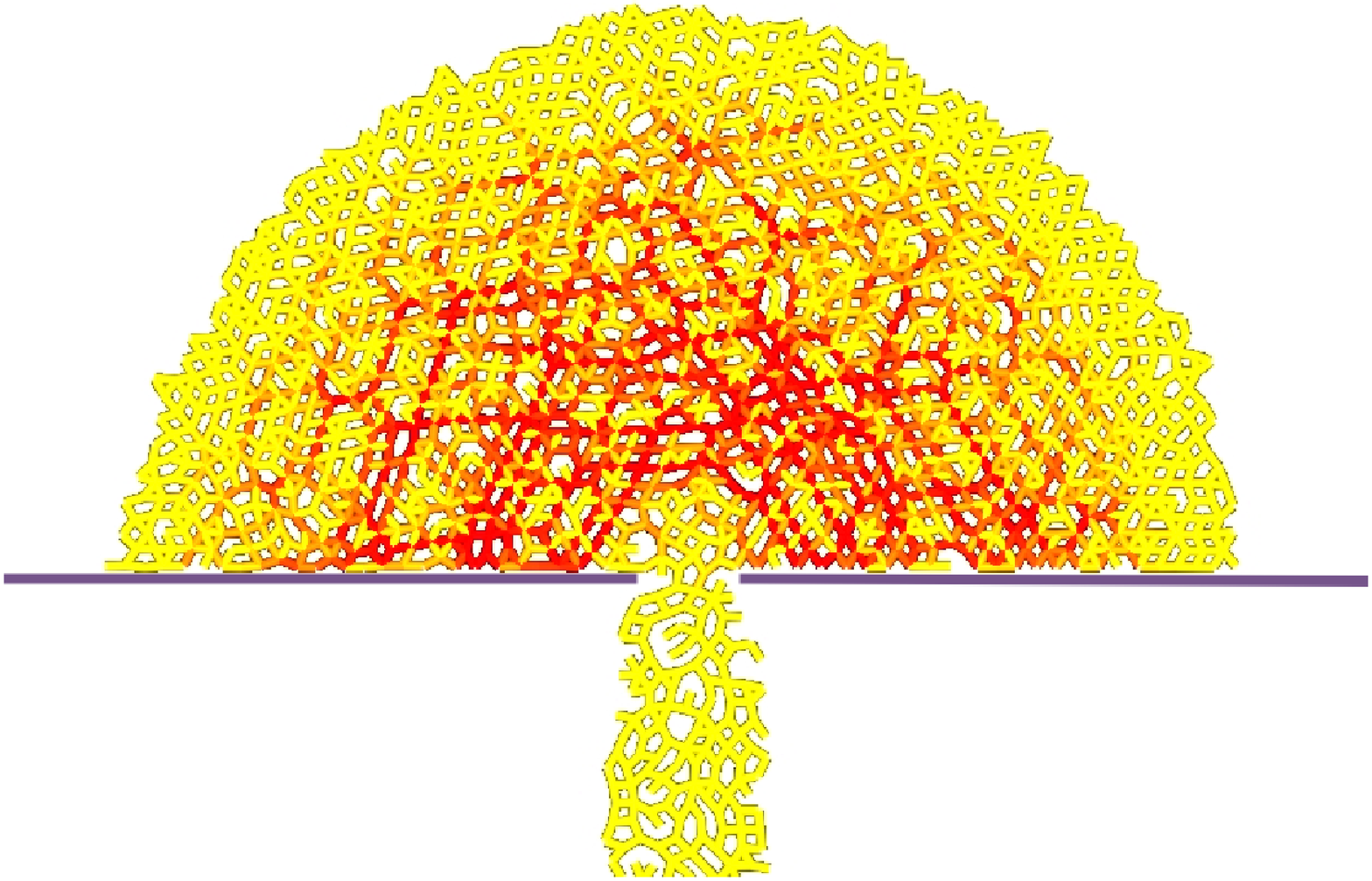}
 \includegraphics[height=0.75\linewidth]{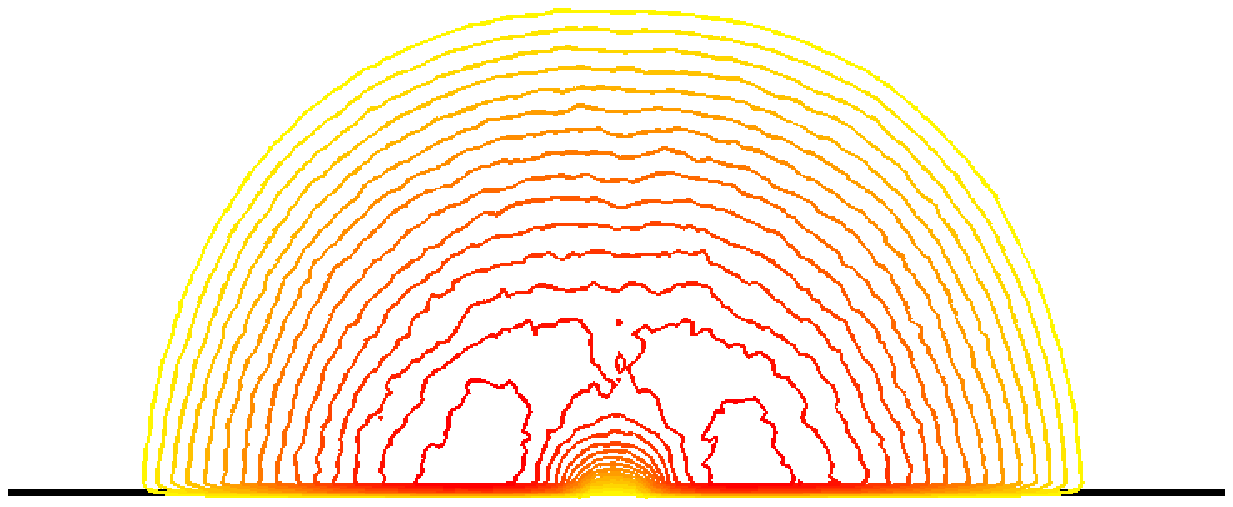}
\caption{Pressure field in the microscopic and macroscopic models.}
\label{fig:pressions}
\end{center}
\end{figure}

\subsection*{Realistic geometries}
Let us finally present some examples which correspond to more realistic evacuations situations. The microscopic setting, as suggested in Section~\ref{sec:micro_macro}, describes more properly  situations where people leave through narrow exits. Fig.~\ref{fig:bat_maths} shows the evacuation of the first floor of a large  building (Departement of Mathematics at Orsay).

\begin{figure}[h!]
\begin{center}
 \includegraphics[width=\linewidth]{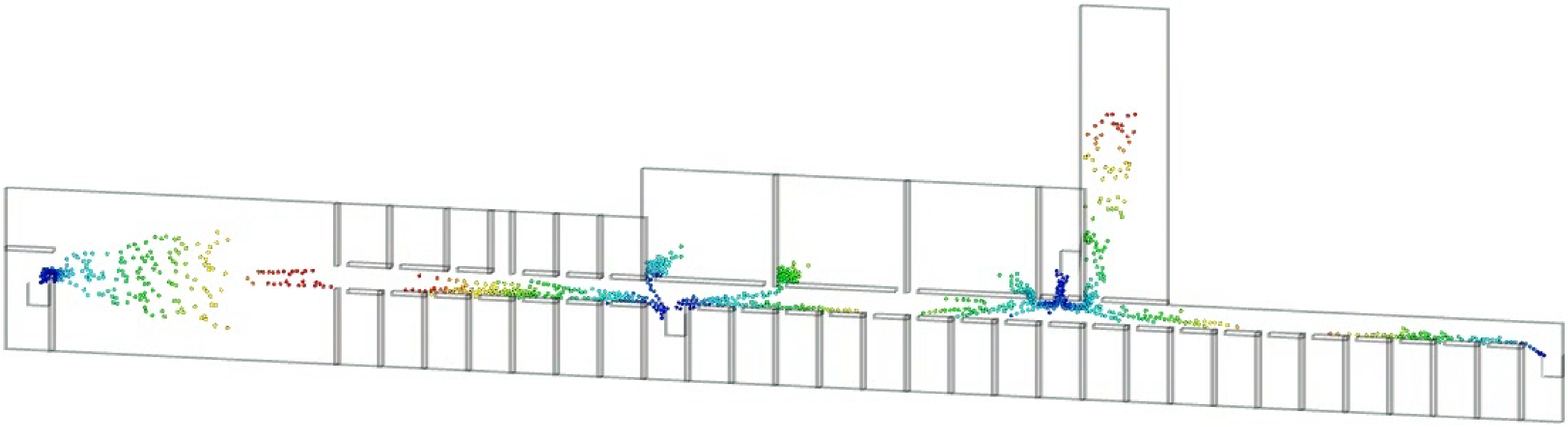}\\
 \begin{minipage}[c]{0.2\linewidth}
 \includegraphics[width=\linewidth]{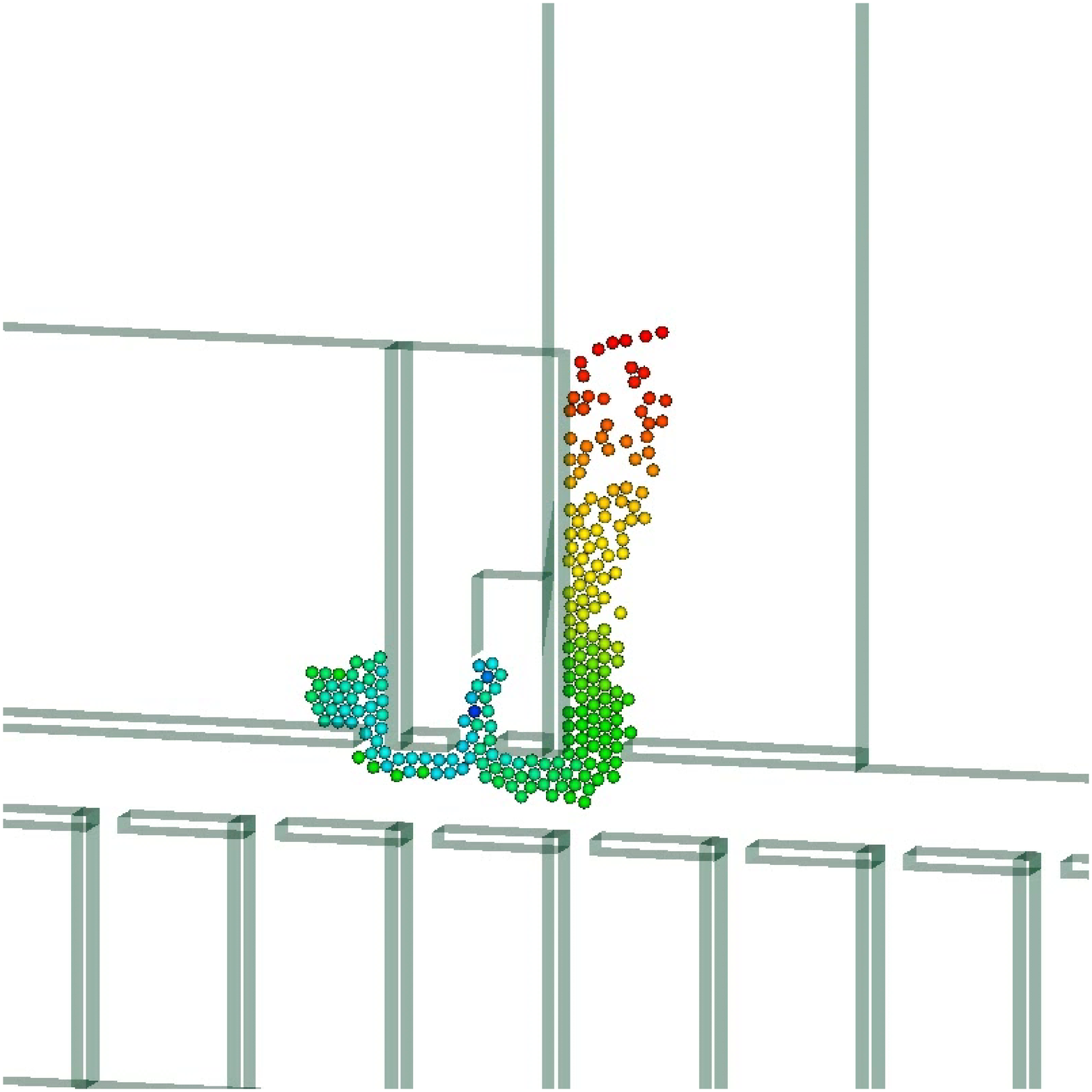}
 \end{minipage}
\hspace*{0.2cm} \begin{minipage}[c]{0.2\linewidth}
 \includegraphics[width=\linewidth]{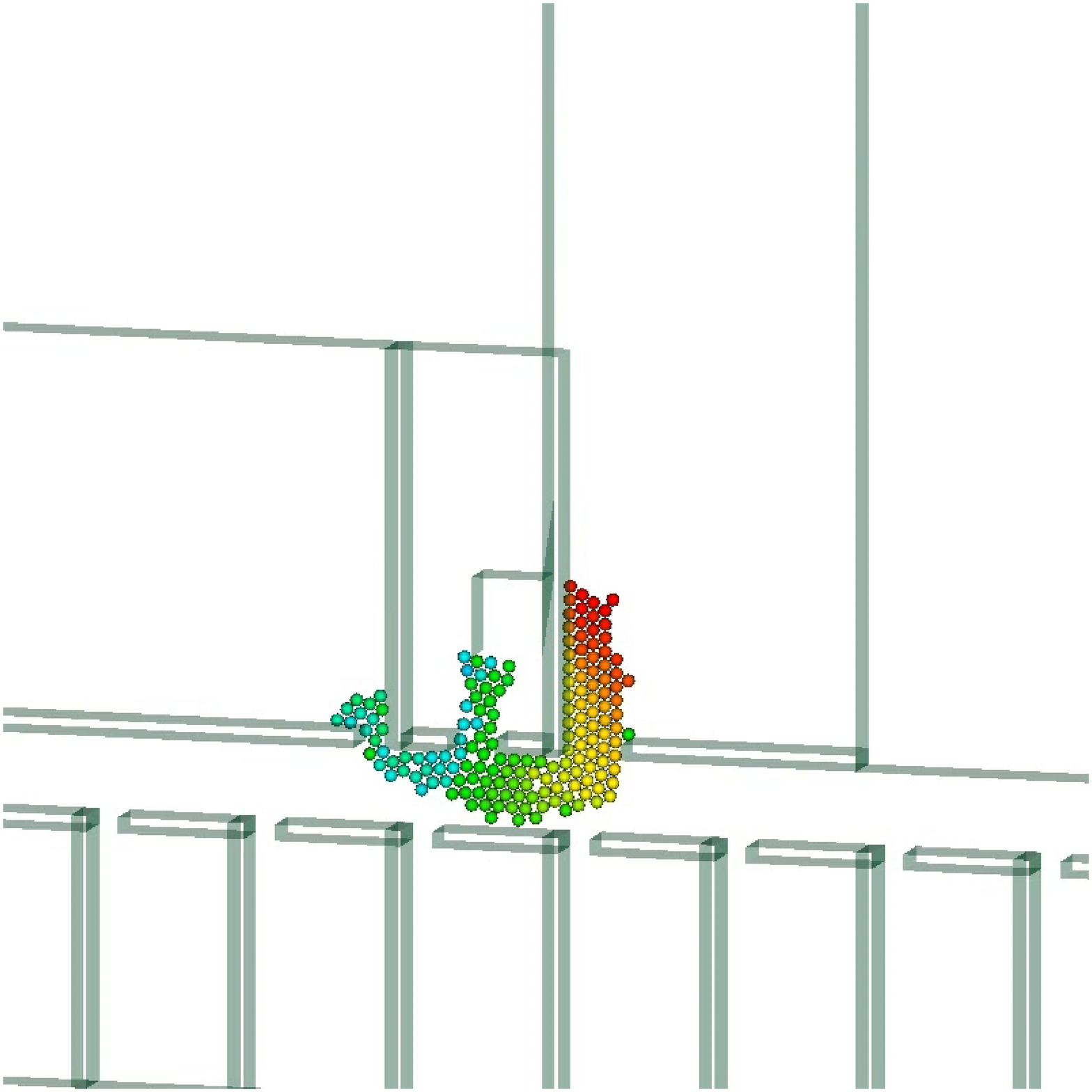}
 \end{minipage}
 \hspace*{0.2cm}  \begin{minipage}[c]{0.2\linewidth}
 \includegraphics[width=\linewidth]{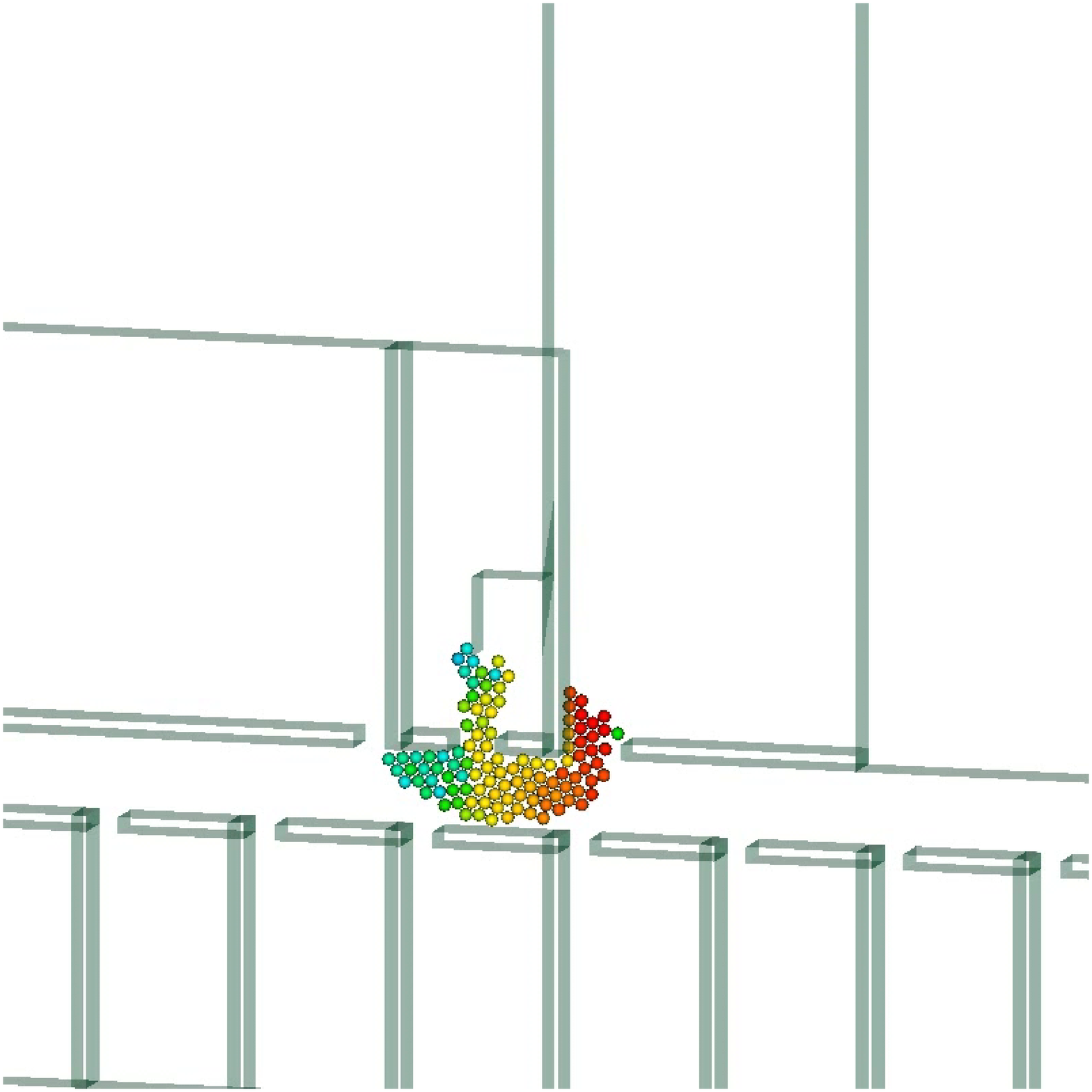}
 \end{minipage}
\hspace*{0.2cm} \begin{minipage}[c]{0.2\linewidth}
 \includegraphics[width=\linewidth]{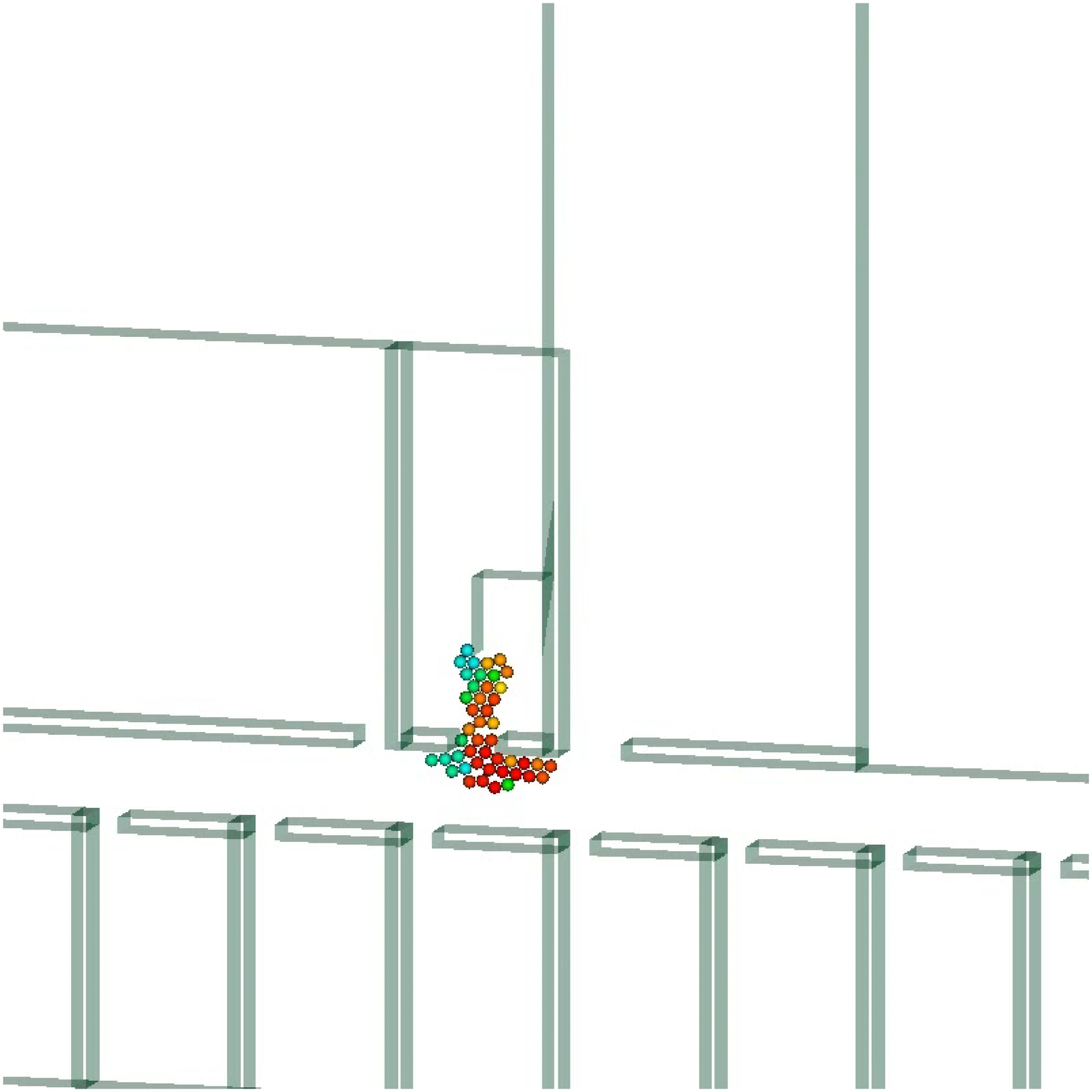}
 \end{minipage}
\caption{Evacuation of the first floor of the Maths building at  Orsay.}
\label{fig:bat_maths}
\end{center}
\end{figure}

The macroscopic setting, on the other hand, would best fit situations where many people have to evacuate a large domain. We present in Fig.~\ref{fig:stade} the evacuation of the Stade de France. Even if the benches of the stadium are quite narrow passages, the macroscopic model gives results that correspond to real evacuation configurations.

\begin{figure}[h!]
\begin{center}
 \includegraphics[width=0.9\linewidth]{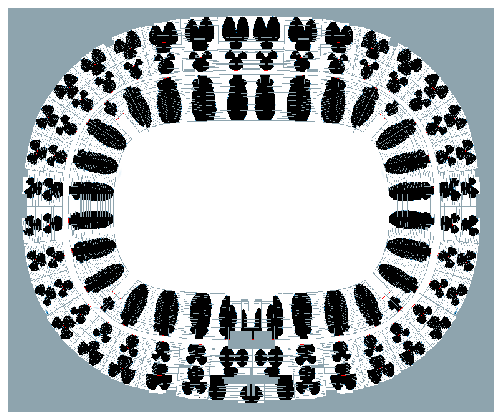}\\
 \begin{minipage}[c]{0.2\linewidth}
 \includegraphics[width=\linewidth]{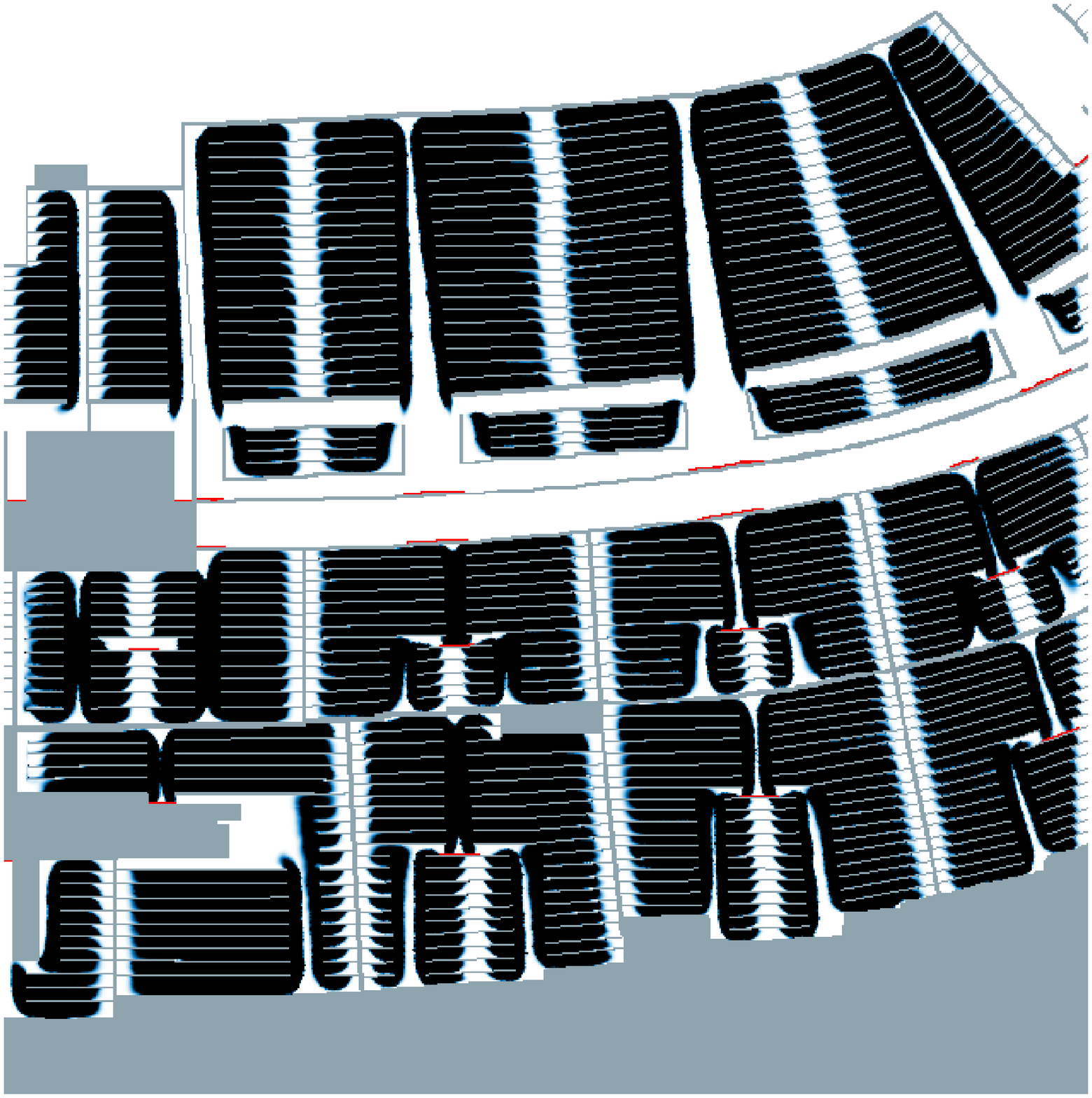}
 \end{minipage}
\hspace*{0.2cm} \begin{minipage}[c]{0.2\linewidth}
 \includegraphics[width=\linewidth]{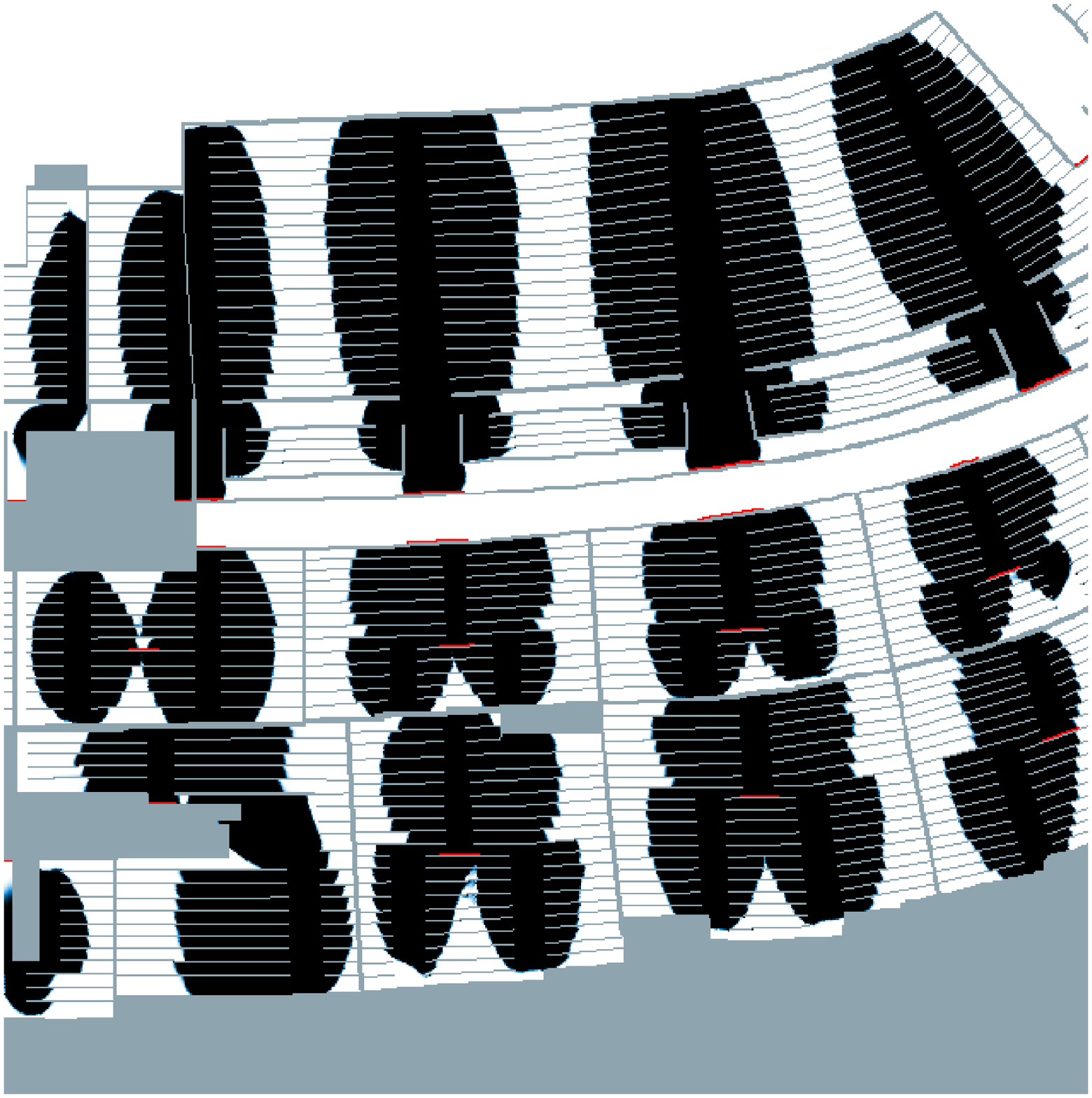}
 \end{minipage}
 \hspace*{0.2cm}  \begin{minipage}[c]{0.2\linewidth}
 \includegraphics[width=\linewidth]{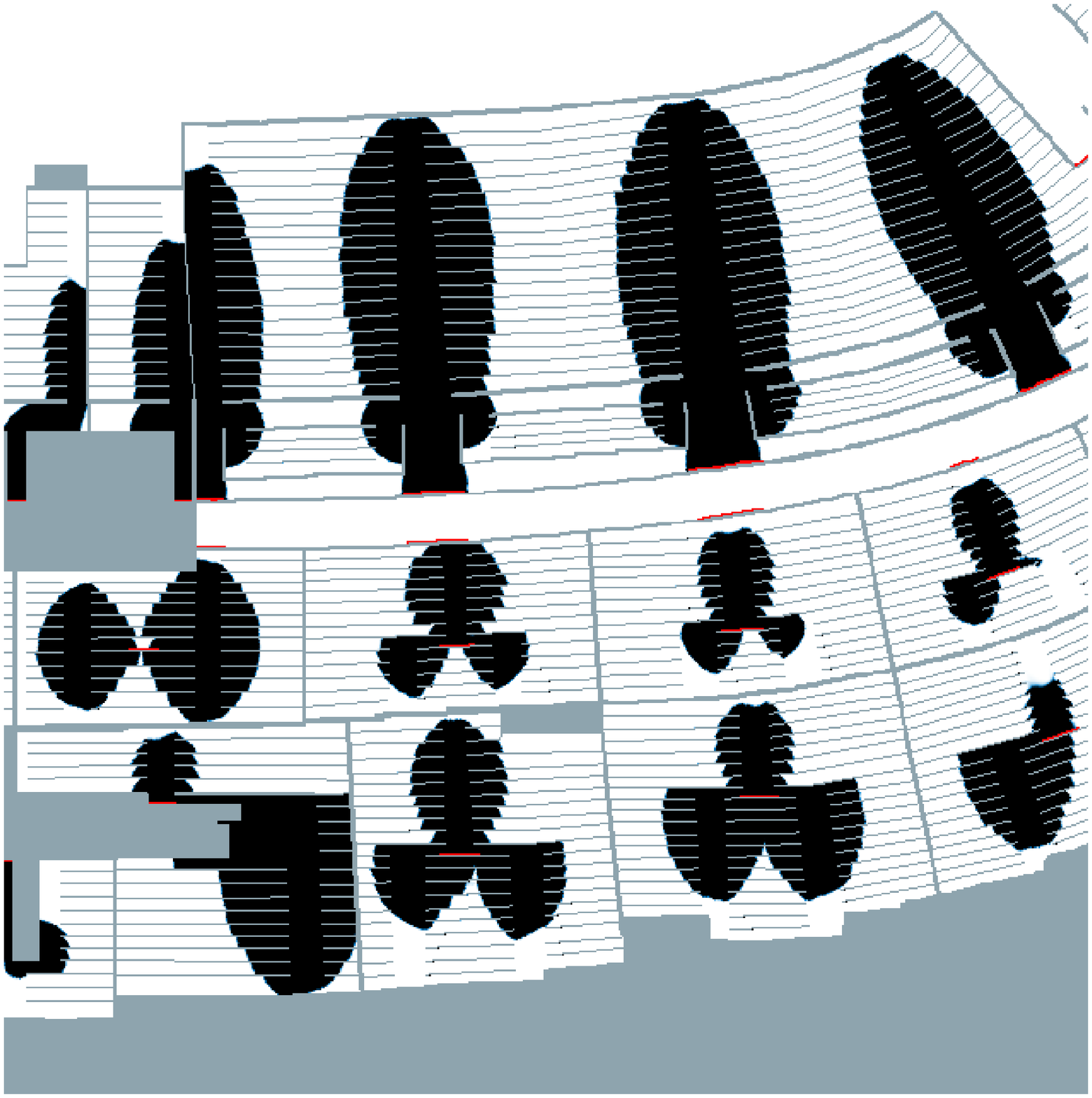}
 \end{minipage}
\hspace*{0.2cm} \begin{minipage}[c]{0.2\linewidth}
 \includegraphics[width=\linewidth]{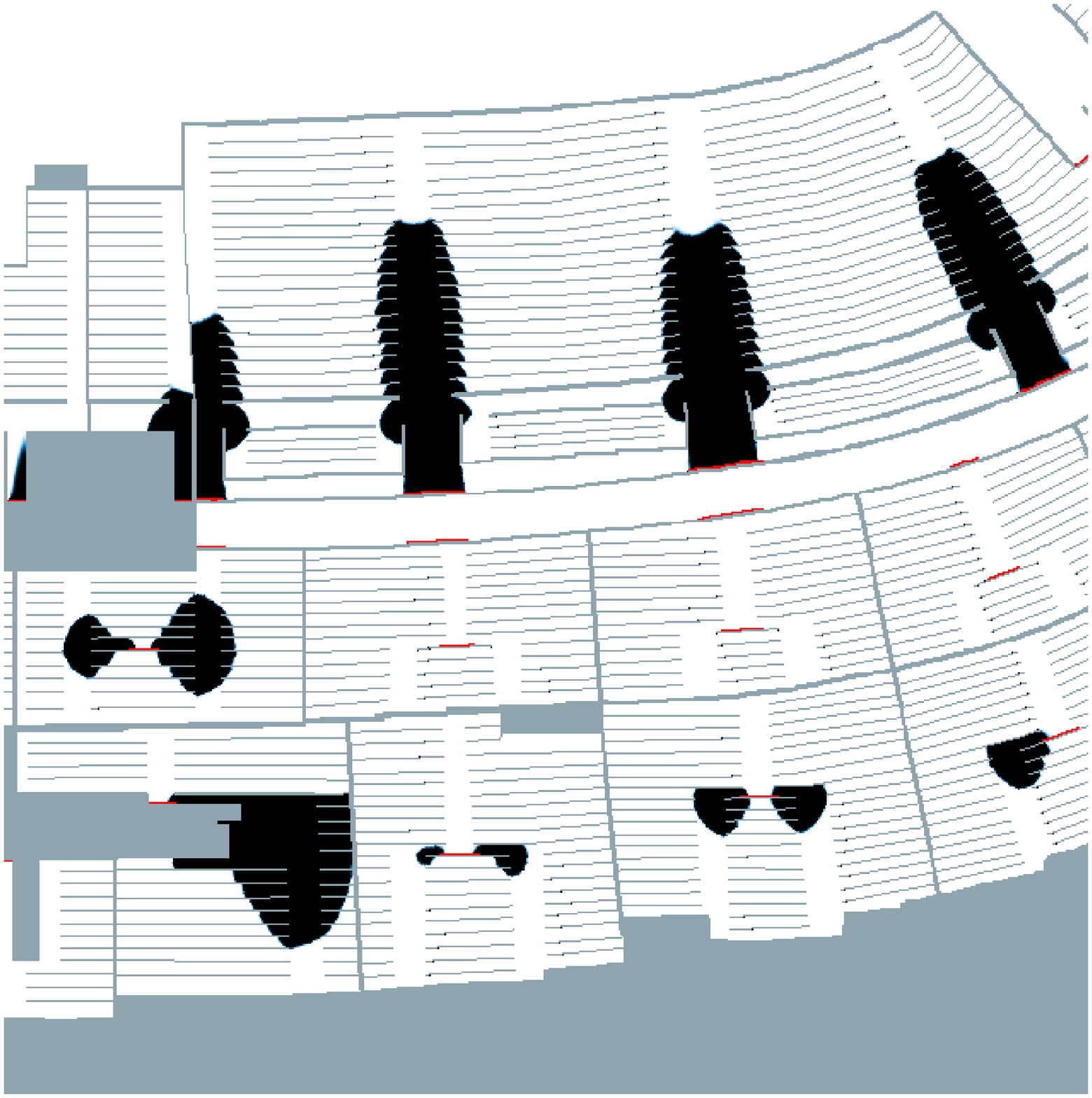}
 \end{minipage}
\caption{Evacuation of the Stade de France.}
\label{fig:stade}
\end{center}
\end{figure}

 \section*{Acknowledgements} 
 We would like to thank   
 E. Maurel-Segala  and L. Thibault
 for fruitful  suggestions and comments.



\def \authorstyle { }

\bibliographystyle{plain}
\footnotesize{
}
%

%






\end{document}